\documentclass[11pt]{amsart}
\usepackage{amsfonts, amsmath, amssymb, color}
\usepackage{amsthm}
\usepackage{booktabs}
\usepackage{float}
\usepackage{hhline}
\usepackage{lipsum}
\usepackage{lscape}
\usepackage{subfig}
\usepackage{textcomp}
\usepackage{tikz}
\usetikzlibrary{decorations.pathmorphing,shapes,arrows,positioning}
\usepackage{xcolor}
\usepackage{young}
\usepackage{youngtab}
\usepackage{ytableau}
\usepackage{rotating} %ÓÃÓÚÐýת
\usepackage{hyperref}
\usepackage{soul}

\soulregister\cite7
\hypersetup{colorlinks,linkcolor=blue,urlcolor=cyan,citecolor=blue}

\allowdisplaybreaks[4]
\theoremstyle{plain}
\newtheorem{thm}{Theorem}[section]
\newtheorem{lem}[thm]{Lemma}
\newtheorem{prop}[thm]{Proposition}
\newtheorem{cor}[thm]{Corollary}
\newtheorem{rem}[thm]{Remark}

\theoremstyle{definition}

\newtheorem{exmp}[thm]{Example}

\newcommand{\la}{\lambda}

\newcommand{\tabincell}[2]{\begin{tabular}{@{}#1@{}}#2\end{tabular}}
\newcommand{\sbtr}{\mathrm{sbtr}}

\numberwithin{equation}{section} \errorcontextlines=0

\begin{document}
\title{Murnaghan-Nakayama rule and spin bitrace for the Hecke-Clifford Algebra}
\author{Naihuan Jing}
\address{Department of Mathematics, North Carolina State University, Raleigh, NC 27695, USA}
\email{jing@ncsu.edu}
\author{Ning Liu}
\address{School of Mathematics, South China University of Technology,
Guangzhou, Guangdong 510640, China}
\email{mathliu123@outlook.com}
%\thanks{{\scriptsize
%\hskip -0.4 true cm MSC (2010): Primary: 81P40; Secondary: 81Qxx.
%$*$Corresponding author, jing@ncsu.edu}}
\subjclass[2020]{Primary: 20C08, 15A66; Secondary: 17B69, 20C15, 05E10}\keywords{Hecke-Clifford algebra, vertex operators, Schur's $Q$-polynomials, Murnaghan-Nakayama rule, bitrace}

\maketitle

\begin{abstract}
A Pfaffian-type Murnaghan-Nakayama rule is derived for the Hecke-Clifford algebra $\mathcal{H}^c_n$ based on the Frobenius formula and vertex operators,
and this leads to a combinatorial version via the tableaux
realization of Schur's $Q$-functions.
As a consequence, a general formula for the irreducible characters $\zeta^{\la}_{\mu}(q)$ using partition-valued functions is derived. Meanwhile, an iterative formula on the indexing partition $\la$ via the Pieri rule is also deduced. As applications, some compact formulae of the irreducible characters
are given for special partitions and a symmetric property of the irreducible character is found. We also introduce the spin bitrace as the analogue of the bitrace for the Hecke algebra and derive its general combinatorial formula. Tables of irreducible characters are listed for $n\leq7.$
\end{abstract}
\tableofcontents

\section{Introduction}
The Hecke algebra associated to the symmetric group controls the representation theory of the quantum general linear group by the Schur-Jimbo duality. Ram \cite{R} used this duality to show the Frobenius character formula for the Hecke algebra and gave a combinatorial rule for computing the irreducible characters. This combinatorial rule is a $q$-extension of the Murnaghan-Nakayama formula for irreducible characters of the symmetric group, which also holds for
Iwahori-Hecke algebras of classical types and the complex reflection groups $G(r,p,n)$ \cite{HR1, HR2}.
 %obtained the Murnaghan-Nakayama rules for characters of.
 As a split semisimple algebra, the Hecke algebra has a second orthogonality relation called the bitrace \cite{R2}.
 %There is an analogue of the second orthogonality relation for characters of a finite group for any finite dimensional algebra with a non-degenerate trace form, including any split semisimple algebra
  %For the Hecke algebra, this second orthogonality relation for characters is called the bitrace.
  A general combinatorial formula of the bitrace at elements of particular shapes was given in \cite{HLR} based on a recursive rule for the Kazhdan-Lusztig polynomials \cite{Ro}. A simple derivation of the bitrace formula has been found \cite{JL} using techniques of vertex operators.

As a q-deformation of the Sergeev algebra $\mathfrak{H}^c_n$, the Hecke-Clifford algebra $\mathcal{H}^c_n$ was defined by Olshanski \cite{O} as a semidirect product of the Hecke algebra $H_n$ by the Clifford algebra $Cl_n$. When $q$ is generic, it satisfies the Schur-Sergeev-Olshanski super-duality with the quantum enveloping algebra of queer Lie superalgebras $q_n$. Moreover, the Grothendieck group of the tower of Hecke-Clifford algebras is isomorphic to the subalgebra of symmetric functions spanned by Schur's Q-functions, which parallels to the classical case of the Sergeev algebra $\mathfrak{H}^c_n$ $(q=1)$ \cite{S}. Jones and Nazarov \cite{JN} constructed spin analogues of Young's symmetrizers for the projective representations. In \cite{WW}, Wan and Wang used the super-duality to derive a Frobenius character formula for the Hecke-Clifford algebra in terms of spin Hall-Littlewood functions.

Representations of affine Lie algebras are closely intertwined with vertex operator algebras. Two realizations of the affine Kac-Moody Lie algebra
$\widehat{sl}(2)$ are given by the so-called untwisted vertex representation and twisted vertex representation, and their Fock spaces can be essentially constructed respectively by the ring of symmetric functions (spanned by Schur functions) and the subring of odd symmetric functions (spanned by Schur Q-functions)
up to tensoring with the group algebra of the root lattice \cite{FLM}. These form the theoretic background for the vertex realizations of Schur functions \cite{DJKM} and Schur Q-functions \cite{Jing}.

Recently, we have formulated a determinant-type Murnaghan-Nakayama formula for the Hecke algebra and gave an alternative proof of Ram's combinatorial Murnaghan-Nakayama formula \cite{JL} using generalization of untwisted vertex operators. The aim of this paper is to derive a spin analogue of the Murnaghan-Nakayama rule using generalization of twisted vertex operators. We formulate first a Pfaffian-type formula, and then give its combinatorial version with the help of the tableaux realization of Schur's $Q$-functions. This leads to several compact formulas for irreducible characters of the Hecke-Clifford algebra.

We also consider the bitrace of the Hecke-Clifford algebra. The bitrace can be viewed a
deformed orthogonality relation for irreducible characters. Our approach of twisted vertex operators naturally provides operational formulas for the
irreducible characters in terms of matrix coefficients of the vertex operators, which can be effectively calculated by their commutation relations.
We derive a spin version of the Halverson-Luduc-Ram's formula for the bitrace for the Hecke-Clifford algebra, which effectively leads to a combinatorial formula.

The main contents of the paper are as follows. In Section 2, we review necessary terminologies about symmetric functions. Based on the Frobenius formula, we express the irreducible characters of the Hecke-Clifford algebras in terms of the inner product of symmetric functions and provide an algebraic computational method via the vertex operator realization of the spin Hall-Littlewood functions. In particular, a compact formula for two-row characters is given (Theorem \ref{t:two}) and a symmetric property of $\zeta^{\la}_{\mu}(q)$ is also found (Theorem \ref{t:symme.}). In Section 3, starting from the Pieri rule, we deduce an iterative formula (Theorem \ref{t:secondit}).
%As an application, a compact formula for $\zeta^{\lambda}_{(n)}(q)$ derives from this iterative formula (Example \ref{t:low(n)}).
By duality, we obtain a Pfaffian-type Murnaghan-Nakayama formula using the Wick rule (Theorem \ref{t:Pf}). We then derive a combinatorial rule for the Murnaghan-Nakayama formula for the Hecke-Clifford algebra (Corollary \ref{c:comb-MN}). As an immediate consequence, a general formula for $\zeta^{\lambda}_{\mu}(q)$ is given (Corollary \ref{t:general}). The spin bitrace is introduced as the second orthogonality relation for irreducible characters in Section 4, and it is expressed by vertex operators.  Subsequently a general combinatorial formula is derived (Theorem \ref{t:btr}). In particular, the regular character is determined using the general formula (Corollary \ref{t:reg}).

\section{Vertex operator realization of Schur's $Q$-polynomials}\label{S:VO}

A ({\it strict}) {\it partition} $\lambda=(\lambda_1,\lambda_2,\ldots)$ is a (strictly) decreasing sequence of non-negative integers (called parts). A partition $\lambda$ is called {\it odd} if all parts are odd integers. The sum $|\lambda|=\sum_i\lambda_i$ is the weight
 and the number of nonzero parts is the length $l(\lambda)$.
A partition $\lambda$ of weight $n$ is denoted by
$\lambda \vdash n$, and the set of partitions (resp. strict partitions or odd partitions) of weight $n$ will be denoted by $\mathcal{P}_n$ (resp. $\mathcal{SP}_n$ or $\mathcal{OP}_n$). It is well-known that $|\mathcal{OP}_n|=|\mathcal{SP}_n|$ by Euler's identity. For $\lambda\in\mathcal{P}_n$, define
\begin{align*}
\delta(\lambda)=
\begin{cases}
0 &\text {if $l(\lambda)$ is even,}\\
1 &\text {if $l(\lambda)$ is odd.}
\end{cases}
\end{align*}

For any partition $\lambda,$ denote $\epsilon(\lambda)=\frac{l(\lambda)-\delta(\lambda)}{2}=[\frac{l(\la)}2]\in \mathbb{Z_+}$, the largest integer $\leq \frac{l(\la)}2$. $\lambda$ can also be arranged in the increasing order, $\lambda=(1^{m_{1}}2^{m_{2}}\ldots)$,
where $m_{i}=\mathrm{Card}\{\lambda_{j}=i\mid 1\leq j\leq l(\lambda)\}$
is the multiplicity of $i$ in $\lambda$. Setting $z_{\lambda}=\prod_{i\geq 1}i^{m_i(\lambda)}m_i(\lambda)!$, we denote
\begin{align}\label{e:zt}
z_{\lambda}(t)&=\frac{z_{\lambda}}{\prod_{i\geq 1}(1-t^{\lambda_i})}\\
n(\lambda)&=\sum\limits_{i\geq1}(i-1)\lambda_i.
\end{align}
When the finite sequence $\lambda=(\lambda_1,\lambda_2,\ldots)$ of nonnegative integers
is ordered but not necessarily weakly decreasing, $\lambda$ is called a {\it composition} of $n=\sum_i\lambda_i$, denoted as $\lambda\models n$. The length
$l(\lambda)$ is the number of nonzero parts.

A partition $\lambda$ is visualized by its Young diagram consisting of nodes (or boxes sitting at) $(i,j)\in \mathbb{Z}_+^{2}$ such that $1\leq i\leq l(\lambda), 1\leq j\leq \lambda_{i}$. For a strict partition $\lambda$, its {\it shifted diagram} $\lambda^*$ is obtained from the ordinary Young diagram by shifting the $k$th row to the right by $k-1$ squares, for each $k$. If $\lambda$ is a diagram, then an inner corner of $\lambda$ is a node $(i,j)\in \lambda$ whose removal still leaves it that of a partition. The conjugate partition $\la'=(\la_1', \ldots, \la_{\la_1}')$ corresponds to the reflection of
the Young diagram along the diagonal.

Let $\Lambda$ be the ring of symmetric functions in the $x_n$ ($n\in\mathbb N$) over $\mathbb{Q}$, and we will also study the ring $\Lambda_{\mathbb{Z}}$ as a lattice of $\Lambda$. The ring $\Lambda$ has several linear bases indexed by partitions. For each $r>1$, let $p_{r}=\sum x_{i}^{r}$ be the $r$th power-sum. Then $p_{\lambda}=p_{\lambda_{1}}p_{\lambda_{2}}\cdots p_{\lambda_{l}}$ $(\lambda\in\mathcal P)$ form a $\mathbb Q$-basis of $\Lambda$.  Let $\Gamma$ denote the subring of $\Lambda$ generated by the $p_r$ indexed by odd $r$:
\begin{align*}
\Gamma=\mathbb{Q}[p_r: r \quad\text{odd}].
\end{align*}
%The Schur $Q$-functions $Q_{\xi},$ $\xi$ strict, form an orthogonal $\mathbb{Z}$-basis of $\Gamma$ under the inner product
We equip $\Gamma$ with the inner product defined by
\begin{align}
\langle p_{\lambda}, p_{\mu}\rangle=2^{-l(\lambda)}\delta_{\lambda\mu}z_{\lambda}, \quad \lambda,\mu\in \mathcal{OP}.
\end{align}
%In fact, for $\lambda, \mu\in \mathcal{SP}$
The function $p_n$ can be viewed as the left multiplication operator on $\Lambda$, and its adjoint operator
is the differential operator $p_n^* =\frac{n}{2}\frac{\partial}{\partial p_n}$. Here the adjoint $A^*$ of the (homogeneous) linear operator $A$ is defined as usually as the $\mathbb Q$-linear and  anti-involutive operator satisfying
\begin{equation}
\langle Au, v\rangle=\langle u, A^*v\rangle
\end{equation}
for $u, v\in \Gamma$.

The ring $\Gamma$ has a distinguished orthogonal basis of Schur's Q-functions $Q_{\lambda}$ indexed by $\lambda\in\mathcal{SP}$:
\begin{align}
\langle Q_{\lambda}, Q_{\mu}\rangle=2^{l(\lambda)}\delta_{\lambda\mu}.
\end{align}

Recall the vertex operator realization of Schur's $Q$-symmetric functions \cite{Jing}.
Let
$\mathbf Q(z)$ %and its adjoint vertex operator $Q^*(z)$
be the linear map: $\Gamma\longrightarrow \Gamma[[z, z^{-1}]]=\mathbb{Q}[[z,z^{-1}]]\otimes\Gamma$ defined by
\begin{align}\label{e:SchurQop}
\mathbf Q(z)&=\mbox{exp} \left( \sum\limits_{n\geq 1, odd} \dfrac{2}{n}p_nz^{n} \right) \mbox{exp} \left( -\sum \limits_{n\geq 1, odd} \frac{\partial}{\partial p_n}z^{-n} \right)=\sum_{n\in\mathbb Z}Q_nz^{n}.
%\label{e:SchurQop*}
%Q^*(z)&=\mbox{exp} \left(-\sum\limits_{n\geq 1, odd} \dfrac{2}{n}p_nz^{n} \right) \mbox{exp} \left(\sum \limits_{n\geq 1, odd} \frac{\partial}{\partial p_n}z^{-n} \right)=\sum_{n\in\mathbb Z}Q^*_nz^{-n}.
\end{align}
and the adjoint operator $\mathbf Q^*(z)=\mathbf Q(-z)$, i.e. $Q^*_{n}=(-1)^nQ_{-n}$.
Then the operators $Q_n \in \mathrm{End}(\Gamma)$ realize Schur's $Q$-functions as follows.
%: $Q_{\lambda}(x)=Q_{\lambda_1}Q_{\lambda_2}\cdots Q_{\lambda_l}.1$ for any strict partition $\lambda=(\lambda_1,\lambda_2,\ldots,\lambda_l)$.

Let $q_{n}=q_{n}(x;t)$ be the symmetric function defined by the generating series
\begin{align}\label{e:qop}
\mathbf q(z)=\mbox{exp} \left( \sum\limits_{n=1, odd}^{\infty}\frac{2}{n}p_{n}z^{n} \right)=\sum\limits_{n\geq0}q_{n}z^{n}
\end{align}
and define
$q_{\lambda}=q_{\lambda_{1}}q_{\lambda_{2}}\cdots q_{\lambda_{l}}$
for any partition $\lambda$, then the set $\{q_{\lambda}| \lambda\in \mathcal{SP}\}$ also forms a basis of $\Gamma.$ Note that $q_n$ is Schur's $Q$-function associated to the one-row partition $(n)$, explicitly.

\begin{align}
q_n=\sum\limits_{\rho\in\mathcal{OP}_n}\frac{2^{l(\rho)}}{z_{\rho}}p_{\rho}   \quad (n>0).
\end{align}

For any strict partition $\mu=(\mu_{1},\mu_{2},\ldots,\mu_{k})$, the
vertex operator product $Q_{\mu_{1}}Q_{\mu_{2}}\cdots Q_{\mu_{k}}.1$ can be expressed as \cite{Jing}
\begin{align}\label{e:Schur}
Q_{\mu_{1}}Q_{\mu_{2}}\cdots Q_{\mu_{k}}.1=\prod\limits_{i<j}\frac{1-R_{ij}}{1+R_{ij}}q_{\mu_{1}}q_{\mu_{2}}\cdots q_{\mu_{k}}=Q_{\mu}
\end{align}
which is the Schur $Q$-function associated to the strict partition $\mu$.

The following relations of $Q_{n}$ %and $Q^*_{m}$
will be useful in our discussion.
\begin{prop}\label{t:relation} \cite{Jing} The components of $\mathbf Q(z)$ %(or $Q^{*}(z)$)
generate a Clifford algebra:
%Namely, $Q^*_n=(-1)^nQ_{-n}$ and
\begin{align}\label{e:relation}
\{Q_m, Q_n\}&=(-1)^n2\delta_{m,-n},   %=\{Q^*_m, Q^*_n\}
%\{Q_m, Q^*_n\}&=2\delta_{m,n}.
\end{align}
where $\{A, B\}:=AB+BA$. Moreover, $Q_{-n}.1=\delta_{n,0} (n\geq0)$.
\end{prop}

The Hecke-Clifford algebra $\mathcal{H}^c_n$ \cite{O} is the associative superalgebra over the field $\mathbb{C}(t^{\frac{1}{2}})$ with even generators $T_1,\cdots, T_{n-1}$ and odd generators $c_1,\cdots, c_n$ subject to the relations:
\begin{align*}
(T_i-t)(T_i+1)=0, \quad &1\leq i\leq n-1,\\
T_iT_j=T_jT_i, \quad &1\leq i,j\leq n-1, \mid i-j\mid>1,\\
T_iT_{i+1}T_i=T_{i+1}T_iT_{i+1}, \quad &1\leq i\leq n-2,\\
c^2_i=1, c_ic_j=-c_jc_i, \quad &1\leq i\neq j\leq n,\\
T_ic_j=c_jT_i, \quad &j\neq i, i+1, 1\leq i\leq n-1, 1\leq j\leq n,\\
T_ic_i=c_{i+1}T_i, \quad &1\leq i\leq n-1.
\end{align*}

For an (ordered) subset $I=\{i_1, i_2, \ldots, i_k\}\subset [n]=\{1,2,\cdots,n\}$, we denote $C_I=c_{i_1}c_{i_2}...c_{i_k}$, and set
$C_{\emptyset}=1$. The elements $c_i$ generate a Clifford algebra denoted as $Cl_n,$ clearly $\{C_I|I\subset [n]\}$ forms a basis of $Cl_n$.  Then, the set $\{T_{\sigma}C_I \mid \sigma \in S_n, I\subset [n]\}$ forms a linear basis of the algebra $\mathcal{H}^c_n$ (cf. \cite{JN}).

We introduce the symmetric function $g_n=g_n(x; t)$ %$g_n(x;t)$
by its generating function:
\begin{align}\label{e:gop}
\mathbf g(z)=\sum\limits_{n\geq0}g_nz^n=\mbox{exp} \left( \sum\limits_{n\geq 1, odd} \dfrac{2(t^n-1)}{n}p_nz^{n} \right)
%\prod\limits_i\frac{1-zx_i}{1+zx_i}\frac{1+tzx_i}{1-tzx_i},
\end{align}
and then set $\tilde{g}_n=\frac{1}{t-1}g_n$.
For any partition $\mu=(\mu_1,\cdots,\mu_l)$, we define $g_{\mu}=g_{\mu_1}\cdots g_{\mu_l}\in \Gamma$ and $\tilde{g}_{\mu}=\tilde{g}_{\mu_1}\cdots\tilde{g}_{\mu_l}$.
It is known that the $g_{\lambda}$ $(\la\in\mathcal{SP})$ form a basis of $\Gamma$.
We also need the adjoint operator $g_n^*$ given by:
%As operators on $\Gamma$, $g_n(x;t)$ and its adjoint are given respectively by:
\begin{align}\label{e:gop2}
%\mathbf g(z)&=\mbox{exp} \left( \sum\limits_{n\geq 1, odd} \dfrac{2(t^n-1)}{n}p_nz^{n} \right)=\sum\limits_{n\geq0}g_nz^n,\\
\mathbf g^*(z)&=\mbox{exp} \left( \sum\limits_{n\geq 1, odd} (t^n-1)\frac{\partial}{\partial p_n}z^{-n} \right)=\sum\limits_{n\geq0}g^*_nz^{-n}.
\end{align}

For $n>0,$ we have explicitly %\eqref{e:SchurQop*} and\eqref{e:gop}-\eqref{e:gop2} are given by
\begin{align}\label{e:qstarn}
Q^*_{n}.1&=\sum\limits_{\rho\in\mathcal{OP}_n}\frac{(-2)^{l(\rho)}}{z_{\rho}}p_{\rho},\\ \label{e:gn}
g_n&=\sum\limits_{\rho\in\mathcal{OP}_n}\frac{(-2)^{l(\rho)}}{z_{\rho}(t)}p_{\rho}.
\end{align}

In \cite{WW}, Wan and Wang established the Frobenius type formula for the irreducible character $\zeta^{\lambda}$ of the Hecke-Clifford algebra $\mathcal{H}^c_n$:
\begin{align}\label{e:characters}
\tilde{g}_{\mu}(x;q)=\sum\limits_{\lambda\in \mathcal{SP}_n}2^{-\frac{l(\lambda)+\delta(\lambda)}{2}}\zeta^{\lambda}_{\mu}(q)Q_{\lambda}(x).
\end{align}
where $\mu\vdash n$ and $Q_{\lambda}$ is the Schur $Q$-function associated with the strict partition $\lambda$.

Therefore for $\la\in\mathcal{SP}, \mu\in\mathcal{P}$ the character value
\begin{align}\label{e:matcoeff}
\zeta^{\lambda}_{\mu}(q)&=2^{-\epsilon(\lambda)}\langle \tilde{g}_{\mu}(x;q), Q_{\lambda}\rangle
=\frac{2^{-\epsilon(\lambda)}}{(q-1)^{l(\mu)}}\langle g_{\mu}(q).1, Q_{\lambda}.1\rangle,
\end{align}
and we are going to compute $G^{\lambda}_{\mu}(t):=\langle g_{\mu}, Q_{\lambda}.1\rangle$ in the following.

For $|z|>|w|,$ it follows from the usual vertex operator calculus that
\begin{align}
\label{e:relations1}
\mathbf g^*(z)\mathbf Q(w)\frac{z+w}{z-w}=\mathbf Q(w)\mathbf g^*(z)\frac{z+tw}{z-tw},
\end{align}
and this implies the following result by expanding the coefficients.
\begin{prop}\label{Up Down}
For any $m,n\in\mathbb{Z},$
\begin{align}\label{e:com1}
g^{*}_{n}Q_{m}+2\sum\limits_{k\geq1}^ng^*_{n-k}Q_{m-k}&=Q_{m}g^{*}_{n}+2\sum\limits_{k\geq1}^{n}t^kQ_{m-k}g^*_{n-k}\\\label{e:com2}
Q_{m}g_{n}+2\sum\limits_{k\geq1}^n(-1)^kQ_{m+k}g_{n-k}&=g_{n}Q_{m}+2\sum\limits_{k\geq1}^{n}(-1)^kt^kg_{n-k}Q_{m+k}.
%Q^*_{m}g_{n}+2\sum\limits_{k\geq1}^nQ^*_{m-k}g_{n-k}=g_{n}Q^*_{m}+2\sum\limits_{k\geq1}^{n}t^kg_{n-k}Q^*_{m-k}.
\end{align}
\end{prop}

For two compositions $\lambda,\mu,$ we say $\lambda\subset \mu$ if $\lambda_{i}\leq \mu_{i}$ for all $i\geq1$. In this case,
we write $\lambda-\mu=(\lambda_{1}-\mu_{1},\lambda_{2}-\mu_{2},\ldots)\vDash |\la|-|\mu|$. For each partition $\la=(\la_1, \ldots, \la_l)$, we define that
\begin{align}
\la^{[i]}=(\la_{i+1}, \cdots, \la_l), \qquad i=0, 1, \ldots, l
\end{align}
So $\la^{[0]}=\la$ and $\la^{[l]}=\emptyset$.

For simplicity we denote for integer $k>0$
\begin{align*}
(k)_t&=\frac{t^k-(-1)^k}{t+1}=t^{k-1}-t^{k-2}+t^{k-3}-\cdots+(-1)^{k-2}t+(-1)^{k-1},\\
[k]_t&=\frac{t^k-1}{t-1}=t^{k-1}+t^{k-2}+\cdots+t+1.
\end{align*}
and we make the convention that $(0)_t=1$ and $(n)_t=0$ $(n<0)$. Note that $(k)_t=(-1)^{k-1}[k]_{-t}$.
For a composition $\tau$, we denote $(\tau)_t=\prod\limits_{a\geq1}(\tau_a)_t$.
The following result can be easily shown by induction.
\begin{lem}\label{t:idetity}
Let $k$ be a positive integer, then we have
\begin{align}
(k)_t+2\sum\limits_{i=1}^{k-1}(i)_t=[k]_t.
\end{align}
\end{lem}
Now we can give the following key result.

\begin{thm}\label{t:iterative}
For $\lambda\in\mathcal{SP}_{n}$, $\mu\in\mathcal{P}_n$ and integer $k$, we have that
\begin{align}\label{e:gQ}
g_{k}^{*}Q_{\lambda}.1&=\sum\limits_{\tau\models k}2^{l(\tau)}(t-1)^{l(\tau)}(\tau)_tQ_{\lambda-\tau}.1=
\sum_{\tau\models k}f_{\tau}Q_{\lambda-\tau}.1,\\\label{e:Qg}
Q_{-k}g_{\mu}.1&=\sum_{i=k}^{n}\sum_{\tau\in \mathcal{C}^{\mu}_{i}}(-1)^if_{\tau}g_{\mu-\tau}Q_{-k+i}.1,
%\sum\limits_{i=k}^{n}\sum\limits_{\tau\in \mathcal{C}^{\mu}_{i}}(-1)^i2^{l(\tau)}(t-1)^{l(\tau)}(\tau)_tg_{\mu-\tau}Q_{-k+i}.1
%Q^*_{k}g_{\mu}.1&=\sum\limits_{i=k}^{n}\sum\limits_{\tau\in \mathcal{C}^{\mu}_{i}}2^{l(\tau)}(t-1)^{l(\tau)}(\tau)_tg_{\mu-\tau}Q^*_{k-i}.1,
\end{align}
where $\mathcal{C}^{\mu}_{k}\triangleq\{\tau\models k\mid \tau\subset\mu\}$ and $f_{\tau}=2^{l(\tau)}(t-1)^{l(\tau)}(\tau)_t$.
\end{thm}
The coefficients $f_{\tau}$ will be computed explicitly in Lemma \ref{l:z}.
\begin{proof}
As the arguments for both identities are similar, we only verify the first one by induction on $k+n$. The initial step is clear. Assume that \eqref{e:gQ} holds for any $k^{'}$ and strict partition $\nu$ with $k^{'}+|\nu|<k+n$,
it follows from Proposition \ref{Up Down} and the inductive hypothesis that
\begin{align*}
g_{k}^{*}Q_{\lambda}.1&=-2\sum\limits_{i=1}^{k}g^*_{k-i}Q_{\lambda_1-i}Q_{\lambda^{[1]}}.1+Q_{\lambda_1}g^*_kQ_{\lambda^{[1]}}.1+2\sum\limits_{i=1}^kt^iQ_{\lambda_1-i}g
^*_{k-i}Q_{\lambda^{[1]}}.1\\
&=-2\sum\limits_{i=1}^k\sum\limits_{\tau\models k-i}2^{l(\tau)}(t-1)^{l(\tau)}\prod\limits_{a\geq1}(\tau_a)_tQ_{(\lambda_1-i,\lambda_2,\cdots\lambda_l)-\tau}.1\\
&+Q_{\lambda_1}\sum\limits_{\tau\models k}2^{l(\tau)}(t-1)^{l(\tau)}\prod\limits_{a\geq1}(\tau_a)_tQ_{\lambda^{[1]}-\tau}.1\\
&+2\sum\limits_{i=1}^kt^iQ_{\lambda_1-i}\sum\limits_{\tau\models k-i}2^{l(\tau)}(t-1)^{l(\tau)}\prod\limits_{a\geq1}(\tau_a)_tQ_{\lambda^{[1]}-\tau}.1
\end{align*}
Next, we consider the first term $A=\sum\limits_{i=1}^k\sum\limits_{\tau\models k-i}2^{l(\tau)}(t-1)^{l(\tau)}\prod\limits_{a\geq1}(\tau_a)_tQ_{(\lambda_1-i,\lambda_2,\cdots\lambda_l)-\tau}.1$. We divided $A$ into two parts : $\tau_1=0$ and $\tau_1\neq 0$.
\begin{align*}
A&=\sum\limits_{i=1}^k\sum\limits_{\tau\models k-i}2^{l(\tau)}(t-1)^{l(\tau)}\prod\limits_{a\geq1}(\tau_a)_tQ_{\lambda_1-i}Q_{\lambda^{[1]}-\tau}.1\\
&+\sum\limits_{i=1}^k\sum\limits_{j=1}^{k-i}\sum\limits_{\tau\models k-i-j}2^{l(\tau)+1}(t-1)^{l(\tau)+1}\prod\limits_{a\geq1}(\tau_a)_t(j)_tQ_{\lambda_1-i-j}Q_{\lambda^{[1]}-\tau}.1
\end{align*}
In the second summand, fixing $i$ and replacing dummy index $j$ by $i+j$ and then moving $(j-i)$ forward, we have that
\begin{align*}
A&=\sum\limits_{i=1}^k\sum\limits_{\tau\models k-i}2^{l(\tau)}(t-1)^{l(\tau)}\prod\limits_{a\geq1}(\tau_a)_tQ_{\lambda_1-i}Q_{\lambda^{[1]}-\tau}.1\\
&+\sum\limits_{i=1}^k\sum\limits_{j=i+1}^{k}(j-i)_t\sum\limits_{\tau\models k-j}2^{l(\tau)+1}(t-1)^{l(\tau)+1}\prod\limits_{a\geq1}(\tau_a)_tQ_{\lambda_1-j}Q_{\lambda^{[1]}-\tau}.1\\
&=\sum\limits_{i=1}^k\sum\limits_{\tau\models k-i}2^{l(\tau)}(t-1)^{l(\tau)}\prod\limits_{a\geq1}(\tau_a)_tQ_{\lambda_1-i}Q_{\lambda^{[1]}-\tau}.1\\
&+\sum\limits_{j=2}^k(t-1)([j]_t-(j)_t)\sum\limits_{\tau\models k-j}2^{l(\tau)}(t-1)^{l(\tau)}\prod\limits_{a\geq1}(\tau_a)_tQ_{\lambda_1-j}Q_{\lambda^{[1]}-\tau}.1~~~\text{(by Lemma \ref{t:idetity})}\\
&=\sum\limits_{\tau\models k-1}2^{l(\tau)}(t-1)^{l(\tau)}\prod\limits_{a\geq1}(\tau_a)_tQ_{\lambda_1-1}Q_{\lambda^{[1]}-\tau}.1\\
&+\sum\limits_{j=2}^kt^j\sum\limits_{\tau\models k-j}2^{l(\tau)}(t-1)^{l(\tau)}\prod\limits_{a\geq1}(\tau_a)_tQ_{\lambda_1-j}Q_{\lambda^{[1]}-\tau}.1\\
&-\sum\limits_{j=2}^k\sum\limits_{\tau\models k-j}2^{l(\tau)}(t-1)^{l(\tau)+1}\prod\limits_{a\geq1}(\tau_a)_t(j)_tQ_{\lambda_1-j}Q_{\lambda^{[1]}-\tau}.1.
\end{align*}
Plugging $A$ into the original equation, we then have
\begin{align*}
g_{k}^{*}Q_{\lambda}.1&=-2\sum\limits_{\tau\models k-1}2^{l(\tau)}(t-1)^{l(\tau)}\prod\limits_{a\geq1}(\tau_a)_tQ_{\lambda_1-1}Q_{\lambda^{[1]}-\tau}.1\\
&+2\sum\limits_{j=2}^k\sum\limits_{\tau\models k-j}2^{l(\tau)}(t-1)^{l(\tau)+1}\prod\limits_{a\geq1}(\tau_a)_t(j)_tQ_{\lambda_1-j}Q_{\lambda^{[1]}-\tau}.1\\
&+Q_{\lambda_1}\sum\limits_{\tau\models k}2^{l(\tau)}(t-1)^{l(\tau)}\prod\limits_{a\geq1}(\tau_a)_tQ_{\lambda^{[1]}-\tau}.1\\
&+2tQ_{\lambda_1-1}\sum\limits_{\tau\models k-1}2^{l(\tau)}(t-1)^{l(\tau)}\prod\limits_{a\geq1}(\tau_a)_tQ_{\lambda^{[1]}-\tau}.1\\
&=(2t-2)\sum\limits_{\tau\models k-1}2^{l(\tau)}(t-1)^{l(\tau)}\prod\limits_{a\geq1}(\tau_a)_tQ_{\lambda_1-1}Q_{\lambda^{[1]}-\tau}.1\\
&+\sum\limits_{j=2}^k\sum\limits_{\tau\models k-j}2^{l(\tau)+1}(t-1)^{l(\tau)+1}\prod\limits_{a\geq1}(\tau_a)_t(j)_tQ_{\lambda_1-j}Q_{\lambda^{[1]}-\tau}.1\\
&+\sum\limits_{\tau\models k}2^{l(\tau)}(t-1)^{l(\tau)}\prod\limits_{a\geq1}(\tau_a)_tQ_{\lambda_1}Q_{\lambda^{[1]}-\tau}.1\\
&=\sum\limits_{j=1}^k\sum\limits_{\tau\models k-j}2^{l(\tau)+1}(t-1)^{l(\tau)+1}\prod\limits_{a\geq1}(\tau_a)_t(j)_tQ_{\lambda_1-j}Q_{\lambda^{[1]}-\tau}.1\\
&+\sum\limits_{\tau\models k}2^{l(\tau)}(t-1)^{l(\tau)}\prod\limits_{a\geq1}(\tau_a)_tQ_{\lambda_1}Q_{\lambda^{[1]}-\tau}.1\\
&=\sum\limits_{\tau\models k}2^{l(\tau)}(t-1)^{l(\tau)}\prod\limits_{a\geq1}(\tau_a)_tQ_{\lambda-\tau}.1.
\end{align*}
\end{proof}

We remark that moving $g_{\mu_l}$ to the right side can simplify computation. In view of the proof Theorem \ref{t:iterative} still holds for compositions $\lambda, \mu\in\mathbb{Z}_{+}^n.$
\begin{exmp}
For $\lambda=(6,2,1), \mu=(5,3,1)$.
\begin{align*}
G^{\lambda}_{\mu}(t)&=\langle g_{5}g_{3}g_{1}, Q_{6}Q_{2}Q_{1}.1 \rangle\\
&=2(t-1)(\langle g_{5}g_{3}, Q_{5}Q_{2}Q_{1}.1 \rangle+\langle g_5g_3, Q_6Q_2.1 \rangle)\\
&=4(t-1)^2(-2\langle g_{5}, Q_{5}.1 \rangle+4(t-1)^2\langle g_5, Q_4Q_1.1 \rangle+(3t^2-5t+3)\langle g_5, Q_3Q_2.1 \rangle)\\
&=-16(t-1)^4(4t^4-10t^3+10t^2-4t-1).
\end{align*}
Therefore, $\zeta^{\lambda}_{\mu}(q)=-8(q-1)(4q^4-10q^3+10q^2-4q-1).$
\end{exmp}

A special case of \eqref{e:gQ} is the following.
\begin{exmp}\label{t:(n)}
Let $\mu\in\mathcal{OP}_n$, we have
\begin{align}
\zeta^{(n)}_{\mu}(q)=2^{l(\mu)}(\mu)_q.
\end{align}
\end{exmp}

\begin{prop} For any strict partition $\la\in\mathcal{SP}_n$, we have that
\begin{equation}
\zeta^{\lambda}_{(1^n)}(q)=2^{n-[\frac{l(\la)}2]}\frac{n!}{\la_1!\cdots\la_l!}\prod_{i<j}\frac{\la_i-\la_j}{\la_i+\la_j},
\end{equation}
where $[a]$ denotes the integer part of the real number $a$.
\end{prop}
\begin{proof} It follows from \eqref{e:matcoeff} and the formula \cite[(6.51)]{Jing} that
\begin{align*}
\zeta^{\la}_{(1^n)}(q)&=\frac{2^{-\epsilon(\zeta)}}{(q-1)^n}\langle g_{(1^n)}, Q_{\la}.1\rangle\\
&=2^{n-\epsilon(\zeta)}\langle p_1^n, Q_{\la}.1\rangle=
\frac{2^{n-\epsilon(\la)}n!}{\la_1!\cdots\la_l!}\prod_{i<j}\frac{\la_i-\la_j}{\la_i+\la_j}.
\end{align*}
\end{proof}

We remark that $\zeta^{\la}_{(1^n)}(q)$ can also be written in terms of the number of shifted standard
Young tableaux associated with $\la$.

To compute the characters of the Hecke-Clifford algebra, we need the following simple identities.
\begin{lem}\label{l:z} For $\mu\in\mathcal{OP}_n$,
\begin{align}\label{e:g.1}
q_n(t, -1)&=g_n(1)=\sum\limits_{\rho \in\mathcal{OP}_n}\frac{(-2)^{l(\rho)}}{z_{\rho}(t)}=
\begin{cases}
2(t-1)(n)_t&\text{if $n\geq 1$}\\
1&\text{if $n=0$}
\end{cases}
\\ \label{e:X(n)}
\langle p_{\mu}, Q_n.1 \rangle&=1.
\end{align}
In particular, $f_{\tau}=q_{\tau}(t, -1)$ for any composition $\tau$.
\end{lem}
\begin{proof} The first equality of \eqref{e:g.1} follows from $\mathbf{q}(z)|_{x=(t, -1, 0, \ldots)}=\mathbf{g}(z)|_{x=(1, 0, \ldots)}$, then
the second equality follows from \eqref{e:gn}.
%from  $g_n(1)=\sum_{\rho\in\mathcal{OP}_n}\frac{(-2)^{l(\rho)}}{z_{\rho}(t)}$.

As for the third identity of \eqref{e:g.1}, note that $\mathbf g(z)=\mathbf q(tz)\mathbf q(-z)=\prod_{i\geq 1}\frac{(1+tx_iz)(1-x_iz)}{(1-tx_iz)(1+x_iz)}$,
so
\begin{align*}
\sum_{n=0}^{\infty}g_n(1)z^n&=\frac{(1+tz)(1-z)}{(1-tz)(1+z)}=1+\frac{2(t-1)}{t+1}\left(\frac1{1-tz}-\frac1{1+z}\right)\\
&=1+2(t-1)\sum_{n=1}^{\infty}\frac{t^n-(-1)^n}{t+1}z^n
\end{align*}
which implies that $g_n(1)=2(t-1)(n)_t$ for $n>0$.

\eqref{e:X(n)} follows from
$\langle p_{\mu}, Q_n.1 \rangle=\langle p_{\mu}, \sum\limits_{\rho\in\mathcal{OP}_n}\frac{2^{l(\rho)}}{z_{\rho}}p_{\rho} \rangle=1$.
\end{proof}

%For any composition $\tau=(\tau_1, \tau_2, \ldots)$, define the $t$-deformed number $f_{\tau}$ by
%\begin{equation}\label{e:f}
%f_{\tau}:=q_{\tau}(t, -1)=\prod_{i\geq 1}q_{\tau_i}(t, -1)=2^{l(\tau)}\prod_{i\geq 1}(\tau_i)_t.
%\end{equation}

For $\mu\in\mathcal{OP}_n$ and $1\leq i\leq n$, we define a sequence of polynomials in $t$:
\begin{align}\label{e:ai}
c_i(\mu;t)&=\sum\limits_{\tau\in \mathcal{C}^{\mu}_{i}}(2t-2)^{l(\tau)+l(\mu-\tau)}(\tau)_t(\mu-\tau)_t=\sum_{\tau\in \mathcal{C}^{\mu}_{i}}f_{\tau}f_{\mu-\tau}.
\end{align}

For any two partitions $\la$ and $\mu$, $\la\cup\mu$ denotes the partition whose parts consist of those of $\lambda$ and $\mu$ rearranged in descending order.
%to be the partition whose parts come from those of $\la$ and $\mu$, arranged in

\begin{thm}\label{t:two}
For $\mu\in\mathcal{OP}_n$ and $k> n-k>0$ we have
\begin{align}\label{e:two}
\zeta^{(k,n-k)}_{\mu}(q)=\frac{1}{2(q-1)^{l(\mu)}}\left(\sum\limits_{i=k+1}^{n}(-1)^{i-k}c_i(\mu;q)+\sum\limits_{i=k}^{n}(-1)^{i-k}c_i(\mu;q)\right).
\end{align}
\end{thm}
\begin{proof}By Theorem \ref{t:iterative} and Lemma \ref{l:z}, we have
\begin{align*}
&G^{(k,n-k)}_{\mu}(t)=\langle g_{\mu}, Q_{(k,n-k).1} \rangle=\langle Q_{k}^{*}g_{\mu}, Q_{n-k}.1 \rangle\\
=&\sum\limits_{i=k}^{n}\sum\limits_{\tau\in\mathcal{C}^{\mu}_{i}}(2t-2)^{l(\tau)}\prod\limits_{a\geq1}(\tau_a)_t\langle g_{\mu-\tau}Q^*_{k-i}.1, Q_{n-k}.1 \rangle\\
=&\sum\limits_{i=k}^{n}\sum\limits_{\tau\in\mathcal{C}^{\mu}_{i}}(2t-2)^{l(\tau)}\prod\limits_{a\geq1}(\tau_a)_t\sum\limits_{\mu^{(j)}\in
\mathcal{OP}_{\mu_{j}-\tau_{j}}}\sum\limits_{\lambda\in\mathcal{OP}_{i-k}}
\frac{(-2)^{l(\mu^{(1)})}\cdots(-2)^{l(\mu^{(l)})}(-2)^{l(\lambda)}}{z_{\mu^{(1)}}(t)\cdots z_{\mu^{(l)}}(t)z_{\lambda}}
\langle p_{\mu^{(1)}\bigcup\cdots\bigcup\lambda}, Q_{n-k}.1 \rangle\\
=&\sum\limits_{i=k}^{n}\sum\limits_{\tau\in\mathcal{C}^{\mu}_{i}}(2t-2)^{l(\tau)}\prod\limits_{a\geq1}(\tau_a)_t\sum\limits_{\mu^{(1)}\in\mathcal{OP}_{ \mu_{1}-\tau_{1}}}\cdots\sum\limits_{\mu^{(l)}\in\mathcal{OP}_{\mu_{l}-\tau_{l}}}\sum\limits_{\lambda\in\mathcal{OP}_{i-k} }\frac{(-2)^{l(\mu^{(1)})}\cdots(-2)^{l(\mu^{(l)})}(-2)^{l(\lambda)}}{z_{\mu^{(1)}}(t)\cdots z_{\mu^{(l)}}(t)z_{\lambda}}\\
=&\sum\limits_{i=k+1}^{n}\sum\limits_{\tau\in\mathcal{C}^{\mu}_{i}}(-1)^{i-k}2(2t-2)^{l(\mu-\tau)+l(\tau)}\prod\limits_{a\geq1}(\tau_a)_t(\mu_a-\tau_a)_t
+\sum\limits_{\tau\in\mathcal{C}^{\mu}_{k}}(2t-2)^{l(\mu-\tau)+l(\tau)}\prod\limits_{a\geq1}(\tau_a)_t(\mu_a-\tau_a)_t.
\end{align*}
The last identity holds by \eqref{e:g.1}. The theorem follows by recalling \eqref{e:ai}. Note that $\zeta_{\mu}^{(k,n-k)}(q)=\frac{1}{2(q-1)^{l(\mu)}}G_{\mu}^{(k,n-k)}(q).$
\end{proof}

To compute $G_{\mu}^{(k,n-k)}(t)$ more efficiently, we use the following generating function in $v$.
\begin{lem} For an odd partition $\mu=(\mu_1, \ldots, \mu_r)$, we have that
\begin{equation}\label{e:av}
C(v)=\sum_{i=0}^{|\mu|}c_i(\mu, t)v^i=(2t-2)^r\prod_{i=1}^{r}\left((\mu_i)_t+2(t-1)\sum_{j=1}^{\mu_i-1}(j)_t(\mu_i-j)_tv^j+(\mu_i)_tv^{\mu_i}\right)
\end{equation}
\end{lem}
\begin{proof} By definition it follows that
\begin{align*}
C(v)&=\sum_{i=0}^{|\mu|}c_i(\mu;t)v^i=\sum_{i=0}^{|\mu|}\sum_{\tau\subset\mu, \tau\vDash i}(2t-2)^{l(\tau)+l(\mu-\tau)}\prod_{a\geq1}(\tau_a)_t(\mu_a-\tau_a)_tv^i\\
&=\sum_{i=0}^{|\mu|}\sum_{\tau_1+\cdots+\tau_r=i}(2t-2)^{r-\delta_{\tau_1,\mu_1}-\cdots -\delta_{\tau_r,\mu_r}+r-\delta_{\tau_1,0}-\cdots -\delta_{\tau_r,0}}\prod_{a\geq1}(\tau_a)_t(\mu_a-\tau_a)_tv^k\\
&=\prod_{i=1}^r\left(\sum_{\tau_i=0}^{\mu_i}(2t-2)^{2-\delta_{\tau_i,\mu_i}-\delta_{\tau_i,0}}(\tau_i)_t(\mu_i-\tau_i)_tv^{\tau_i}\right)\\
&=(2t-2)^r\prod_{i=1}^r\left(\sum_{j=0}^{\mu_i}(2t-2)^{1-\delta_{j,\mu_i}-\delta_{j,0}}(j)_t(\mu_i-j)_tv^{j}\right).
\end{align*}
\end{proof}

Therefore, $$G_{\mu}^{(k,n-k)}(t)=(-1)^k\left([C(v)]_k+[C(v)]_{k+1}\right)\mid_{v=-1}.$$ Here $[C(v)]_k$ means the summands of $C(v)$ with $deg(v)\geq k.$

\begin{exmp}
Let $\lambda=(4,2),\mu=(3,3).$ We can compute $\zeta^{\lambda}_{\mu}(q)$ as follows,
\begin{align*}
C(v)=(2t-2)^2\left( (3)_t+2(t-1)(1)_t(2)_tv+2(t-1)(2)_t(1)_tv^2+(3)_tv^3 \right)\\\left( (3)_t+2(t-1)(1)_t(2)_tv+2(t-1)(2)_t(1)_tv^2+(3)_tv^3 \right).
\end{align*}
So
\begin{align*}
[C(v)]_4&=4(t-1)^2\left((4(3)_t(2)_t(2)_t+4(2)_t(2)_t(2)_t(2)_t)v^4+4(3)_t(2)_t(2)_tv^5+(3)_t(3)_tv^6\right)\\
[C(v)]_5&=4(t-1)^2\left(4(3)_t(2)_t(2)_tv^5+(3)_t(3)_tv^6\right).
\end{align*}
Therefore, $G^{(4,2)}_{(3,3)}(t)=(-1)^4\left([C(v)]_4+[C(v)]_5\right)\mid_{v=-1}=8(t-1)^2(4t^4-4t^3+7t^2-4t+1).$ And $\zeta^{(4,2)}_{(3,3)}(q)=4(q^4-4q^3+7q^2-4q+1).$
\end{exmp}

A polynomial $f(x)\in\mathbb{Z}[x]$ is called {\it symmetric} (resp. {\it antisymmetric}) if there exists $m\in\mathbb{Z}$ such that \begin{align*}
f(x)=x^mf(x^{-1}) \quad ({\rm resp.} f(x)=-x^mf(x^{-1})).
\end{align*}
We claim that $\zeta^{\lambda}_{\mu}(q)$ is symmetric for any $\lambda\in\mathcal{SP}_n,$ $\mu\in\mathcal{OP}_n$. To this we will use the following obvious lemma. %need the following Lemma.
\begin{lem}\label{t:symmetric} Let $f(x), g(x)$ be symmetric or antisymmetric polynomials.\\
(i) $f(x)g(x)$ is symmetric if both of them are symmetric or antisymmetric; It is antisymmetric if one of them is symmetric and another is antisymmetric.\\
(ii) Suppose $deg(f)=deg(g)$ and they are both (anti-)symmetric, then $f(x)\pm g(x)$ is (anti-)symmetric.
\end{lem}
%\begin{proof}
%These results can be checked immediately by definition.
%\end{proof}

\begin{thm}\label{t:symme.}
Let $\lambda\in\mathcal{SP}_n,$ $\mu\in\mathcal{OP}_n$, then the character values $\zeta^{\lambda}_{\mu}(q)$ are
symmetric polynomials in $q$. %then there exists $m_{\lambda\mu}\in\mathbb{Z}$ such that
%\begin{align}
%\zeta^{\lambda}_{\mu}(q)=q^{m_{\lambda\mu}}\zeta^{\lambda}_{\mu}(q^{-1}).
%\end{align}
\end{thm}
\begin{proof}
%Recall
%\begin{align*}
%\tilde{g}^*_{k}Q_{\lambda}.1=\sum_{\tau\models k}2^{l(\tau)}(q-1)^{l(\tau)-1}(\tau)_{q}Q_{\lambda-\tau}.1,
%\end{align*}
%where $\tilde{g}^*_k=\frac{1}{q-1}g^*_k=\frac{1}{q-1}\sum\limits_{\rho\in\mathcal{OP}_k}\frac{(-2)^{l(\rho)}}{z_{\rho}(q)}p^*_{\rho}.$
Repeatedly using \eqref{e:gQ}, we can write that %the process, we have
\begin{align*}
\zeta^{\lambda}_{\mu}(q)=\sum_{\underline{\tau}}2^{l(\underline{\tau})}(q-1)^{l(\underline{\tau})-l(\mu)}(\underline{\tau})_qQ_{\lambda-\underline{\tau}}.1
\end{align*}
where $\underline{\tau}=(\tau^{(1)},\tau^{(2)},\cdots)$ runs through sequences of compositions such that $\tau^{(i)}\models\mu_i$ and $l(\underline{\tau})=\sum_{i} l(\tau^{(i)}),$ $(\underline{\tau})_q=\prod_{i}(\tau^{(i)})_q.$ Here $\la-\underline{\tau}$ means the composition given by $\la-\sum_{i}\tau^{(i)}$.

Note that $2^{l(\tau)}(q-1)^{l(\tau)-1}(\tau)_{q}$ is symmetric by Lemma \ref{t:symmetric} and its degree is $k-1.$ Therefore, $2^{l(\underline{\tau})}(q-1)^{l(\underline{\tau})-l(\mu)}(\underline{\tau})_q$ is symmetric and its degree is $n-l(\mu).$ Subsequently $\zeta^{\lambda}_{\mu}(q)$ is symmetric by using Lemma \ref{t:symmetric} again.
\end{proof}
\section{Murnaghan-Nakayama rule for $\zeta^{\lambda}_{\mu}(q)$}\label{S:Pf-type}
In this section, we will give two iterative formulas for $\zeta^{\lambda}_{\mu}(q).$ We start by recalling some combinatorial notions \cite{Mac}. For two partitions $\mu\subset\lambda,$ the set-theoretic difference $\theta=\lambda-\mu$ is called a {\it skew diagram} denoted by $\lambda/\mu$. Similarly, $\theta^*=\lambda^*-\mu^*$ is called a {\it shifted skew diagram} denoted by $\lambda^*/\mu^*$. A subset $\xi$ of $\theta^*$ is connected if any two squares in $\xi$ are connected by a path in $\xi$. The connected components are the maximal connected subsets of $\theta^*$. Any shifted skew diagram is a union of  connected components, each of which is a shifted skew diagram. A skew diagram $\lambda/\mu$ is called a vertical (resp. horizontal) {\it strip} if each row (resp. column) contains at most one box. A {\it shifted standard tableaux} $T$ of shape $\lambda,$ where $\lambda\in\mathcal{SP}_n$, is a labelling of the squares of $\lambda$ with the integers $1,2,\cdots,n$ such that the marks are %weakly
increasing along each row and down each column. We denote $g^{\lambda}$ the number of shifted standard tableaux of shape $\lambda.$ It is known that \cite[p267]{Mac}
\begin{align*}
g^{\lambda}=\frac{n!}{\la_1!\cdots\la_l!}\prod_{1\leq i<j\leq l}\frac{\la_i-\la_j}{\la_i+\la_j}.
\end{align*}
One can also express $g^{\lambda}$ in terms of hook-lengths of the associated double Young diagram:
$
g^{\lambda}=\frac{n!}{\prod_{x\in \lambda}h(x)},
$
where the hook-length $h(x)$ is the hook-length at $x$ in the double diagram $D(\lambda)=(\lambda_1,\lambda_2,\ldots|\lambda_1-1,\lambda_2-1,\ldots)$ in the Frobenius notation \cite{Sa}.

A {\it shifted tableau} $T$ of shape $\theta^*=\la^*/\mu^*$ is defined to be a sequence of strict partitions $(\la^{(0)},\la^{(1)},\cdots,\la^{(r)})$ such that $\mu=\la^{(0)}\subset\la^{(1)}\subset\cdots\subset\la^{(r)}=\la$, and each $\la^{(i)*}/\la^{(i-1)*}$ is a generalized strip; equivalently, as a labelling of the squares of $\theta^*$ with the integers $1,2,\cdots,r$ which is weakly increasing along rows and down columns and is such that no $2\times 2$ block of squares bears the same label, so that for each $i$ the set of squares labelled $i$ is a generalized strips. The weight (or content) of $T$ is the sequence $\alpha=(\alpha_1,\alpha_2,\cdots, \alpha_r)$ where $\alpha_i$, for each $i\geq1$, is the number of squares in $T$ labelled $i$. Denote by $b(T)$ the number of border strips of which $T$ is composed, and define $x^{\alpha}=x_1^{\alpha_1}\cdots x_r^{\alpha_r}$. By the tableau representation of Schur's Q-functions \cite[p.256]{Mac}
\begin{align}
Q_{\la/\mu}(X)=\sum_{T}2^{b(T)}x^{\alpha}
\end{align}
summed over all shifted tableaux of shape $\theta^*$.
%\hl{We refer the read to} \cite[p256]{Mac} \hl{for the details about the above tableaux realization for skew Schur $Q$-function.}

\subsection{An iterative formula for $\zeta^{\lambda}_{\mu}(q)$ on $\la$.}

We first recall the Pieri rule for Schur Q-functions \cite[(8.15)]{Mac}. Let $\lambda, \mu$ be strict partitions such that $\mu\subset\lambda$ and $\lambda/\mu$ is a horizontal strip. Denote by $a(\lambda/\mu)$ the number of integers $i\geq1$ such that $\lambda/\mu$ has a square in the $i$th column but not in the $(i+1)$th column. Then we have
\begin{align}
2^{-l(\mu)}Q_{\mu}q_{r}=\sum\limits_{\lambda}2^{a(\lambda/\mu)}2^{-l(\lambda)}Q_{\lambda}
\end{align}
where the sum runs over strict partition $\lambda\supset\mu$ such that $\lambda/\mu$ is a horizontal $r$-strip.

%For $\lambda\in\mathcal{SP}_n$, it follows from Proposition \ref{t:relation} that
%\begin{align*}
%\langle Q^*_{-n}.1, Q_{\lambda}.1 \rangle=
%\begin{cases}
%2(-1)^n,& \text{if $\lambda=(n)$}\\
%0,& \text{if others.}
%\end{cases}
%\end{align*}
%Therefore,
%\begin{align}\label{e:Q^*_n}
%Q^*_{-n}.1&=\sum\limits_{\lambda\vdash_sn}\langle Q^*_{-n}.1, Q_{\lambda}.1 \rangle2^{-l(\lambda)}Q_{\lambda}.1=(-1)^nQ_{n}.1.
%\end{align}

\begin{thm}\label{t:secondit}
Let $\lambda\in\mathcal{SP}_n,$ $\mu\in\mathcal{OP}_n,$ we have the following iterative formula for $\zeta^{\lambda}_{\mu}(q):$
\begin{align}\label{e:secondit}
\zeta^{\lambda}_{\mu}(q)=\sum\limits_{i=\lambda_1}^n\sum\limits_{\tau\in\mathcal{C}^{\mu}_{i}}\sum\limits_{\xi}(-1)^{i-\lambda_1}
2^{\epsilon(\xi)-\epsilon(\lambda)+l(\tau)+a(\lambda^{[1]}/\xi)}(q-1)^{l(\tau)+l(\mu-\tau)-l(\mu)}
(\tau)_q\zeta^{\xi}_{\mu-\tau}(q)
\end{align}
where $\xi$ runs over all strict partitions $\xi\subset\lambda^{[1]}$ such that $\lambda^{[1]}/\xi$ are all horizontal $(i-\lambda_1)$-strips.
\end{thm}
\begin{proof}
By \eqref{e:Qg} %and \eqref{e:Q^*_n},
we have
\begin{align*}
G^{\lambda}_{\mu}(t)&=\langle g_{\mu}, Q_{\lambda}.1 \rangle=(-1)^{\lambda_1}\langle Q_{-\lambda_1}g_{\mu}, Q_{\lambda^{[1]}}.1 \rangle\\
%&=\sum\limits_{i=\lambda_1}^n\sum\limits_{\tau\in\mathcal{C}^{\mu}_{i}}2^{l(\tau)}(t-1)^{l(\tau)}(\tau)_t\langle g_{\mu-\tau}Q^*_{\lambda_1-i}.1, Q_{\lambda^{[1]}}.1 \rangle\\
&=\sum\limits_{i=\lambda_1}^n\sum\limits_{\tau\in\mathcal{C}^{\mu}_{i}}2^{l(\tau)}(t-1)^{l(\tau)}(\tau)_t(-1)^{i-\lambda_1}\langle g_{\mu-\tau}Q_{i-\lambda_1}.1, Q_{\lambda^{[1]}}.1 \rangle.
\end{align*}
Recall that $g_{\mu-\tau}=\sum\limits_{\xi}2^{-l(\xi)}G^{\xi}_{\mu-\tau}(t)Q_{\xi}$, then,
\begin{align*}
&G^{\lambda}_{\mu}(t)\\
=&\sum\limits_{i=\lambda_1}^n\sum\limits_{\tau\in\mathcal{C}^{\mu}_{i}}2^{l(\tau)}(t-1)^{l(\tau)}(\tau)_t(-1)^{i-\lambda_1}
\sum\limits_{\xi}2^{-l(\xi)}G^{\xi}_{\mu-\tau}(t)\langle Q_{\xi}Q_{i-\lambda_1}.1, Q_{\lambda^{[1]}}.1\rangle\\
=&\sum\limits_{i=\lambda_1}^n\sum\limits_{\tau\in\mathcal{C}^{\mu}_{i}}\sum\limits_{\xi}2^{l(\tau)+a(\lambda^{[1]}/\xi)}(t-1)^{l(\tau)}
(\tau)_t(-1)^{i-\lambda_1}G^{\xi}_{\mu-\tau}(t)~~~\quad \text{(by Pieri formula)}
\end{align*}
where the sum runs over strict partitions $\xi\subset\lambda^{[1]}$ such that $\lambda^{[1]}/\xi$ are horizontal $(i-\lambda_1)$-strip. And \eqref{e:secondit} holds by $\zeta^{\lambda}_{\mu}(q)=2^{\frac{-l(\lambda)+\delta(\lambda)}{2}}\frac{1}{(q-1)^{l(\mu)}}G^{\lambda}_{\mu}(q).$
\end{proof}

Theorem \ref{t:secondit} implies that $\zeta^{\lambda}_{\mu}(q)$ can be expressed as a linear combination of those character values indexed by partitions of length $l(\lambda)-1$ or $l(\lambda)-2.$
%In particular, $\mu=(n),$ $n$ is odd, \eqref{e:secondit} can be simplified as
%\begin{align}\label{e:(n)it}
%\zeta^{\lambda}_{(n)}(q)=\sum\limits_{i=\lambda_1}^{n}\sum\limits_{\xi}(-1)^{i-\lambda_1}2^{a(\lambda^{[1]}/\xi)+\frac{l(\xi)-\delta(\xi)}{2}}(q-1)^{1-\delta_{n,i}}(i)_q
%\zeta_{(n-i)}^{\xi}(q)
%\end{align}
%where the sum runs over strict partitions $\xi\subset\lambda^{[1]}$ such that $\lambda^{[1]}/\xi$ are horizontal $(i-\lambda_1)$-strips.

\subsection{Computation of $g^{*}_{k}Q_{\lambda}.1$ via Pfaffians.}

Next we will give the second iterative formula for $\zeta^{\lambda}_{\mu}(q)$ which can be viewed as the Murnaghan-Nakayama type rule. Before that we compute the coefficient $C(\lambda/\mu)$ of $Q_{\mu}.1$ in $g^{*}_{k}Q_{\lambda}.1$ via Pfaffians.

Let $A=(a_{ij})$ be a skew symmetric matrix of even size $2n\times2n$, the Pfaffian ${\rm Pf}(A)$ is defined by ${\rm Pf}(A)\doteq\sqrt{{\rm det}(A)}$ and explicitly
\begin{align}
\text{Pf(A)}=\sum\limits_{\sigma} (-1)^{l(\sigma)}a_{\sigma(1)\sigma(2)}\cdots a_{\sigma(2n-1)\sigma(2n)}
\end{align}
summed over all $2$-shuffles of $\{1,2,\cdots,2n\},$ i.e., $ \sigma\in\mathfrak{S}_{2n}$ such that $\sigma(2i-1)<\sigma(2i)$ for $1\leq i\leq n,$ and $\sigma(2j-1)<\sigma(2j+1)$ for $1\leq j\leq n-1.$ The Laplace-type expansion says that
%For antisymmetric matrix $A=(a_{i,j})_{m\times m}$, ${\rm Pf}(A)$ has the Laplace-type expansion:
\begin{align}\label{e:decomposition1}
\text{Pf(A)}=a_{12}\text{Pf($A_{12}$)}-a_{13}\text{Pf($A_{13}$)}+\cdots+(-1)^{m}a_{1m}\text{Pf($A_{1m}$)},
\end{align}
where $A_{ij}$ is the submatrix of $A$ by deleting the $i$th and $j$th rows and columns. More generally, for each fixed $i$
\begin{align}\label{e:decomposition2}
\text{Pf(A)}=(-1)^{i-1}\sum\limits_{j\neq i}(-1)^{j}a_{ij}\text{Pf($A_{ij}$)}.
\end{align}

For two strict partitions $\lambda=(\lambda_1,\lambda_2,\cdots,\lambda_{s}),$ $\mu=(\mu_1,\mu_2,\cdots,\mu_{r})$. We define
the expectation value
\begin{align}
\langle\mu|\lambda\rangle\doteq\langle Q^*_{-\mu_1}Q^*_{-\mu_2}\cdots Q^*_{-\mu_{r}}.1,  Q_{\lambda_1}Q_{\lambda_2}\cdots Q_{\lambda_{s}}.1\rangle=2^{l(\mu)}(-1)^{|\mu|}\delta_{\mu,\lambda}.
\end{align}
In general $\langle \omega_{1}\omega_{2}\cdots \omega_{m}|\omega_{m+1}\omega_{m+2}\cdots \omega_{m+n}\rangle$ is given by Pfaffians through the Wick formula \cite[A.]{DJKM}
%(\cite{vO} $A.4$)
as follows.

Let each of $\omega_i$ be a linear combination of Fermi operators $\psi_j, \psi_j^*$:
\begin{align*}
\omega_i=\sum_{j\in\mathbb{Z}}v_{ij}\psi_j+\sum_{j\in\mathbb{Z}}u_{ij}\psi^*_j, \quad i=1,2,\cdots,m+n.
\end{align*}
Then we have
\begin{align}\label{e:wick}
\begin{split}
&\langle \omega_{1}\omega_{2}\cdots \omega_{m}|\omega_{m+1}\omega_{m+2}\cdots \omega_{m+n}\rangle\\
:=&\langle \omega^{*}_{m}\omega^{*}_{m-1}\cdots \omega^{*}_{1}.1,  \omega_{m+1}\omega_{m+2}\cdots \omega_{m+n}.1\rangle=
\begin{cases}
\text{Pf}({B})& \text{if $n+m$ is even}\\
0& \text{if others}
\end{cases}
\end{split}
\end{align}
where $B$ is the $(n+m)$ by $(n+m)$ antisymmetric matrix with entries $B_{ij}=
\langle \omega^{*}_{i}.1, \omega_{j}.1 \rangle$, $1\leq i< j\leq n+m$.

For a given strict partition $\lambda=(\lambda_1,\lambda_2,\cdots,\lambda_{s})$, by \eqref{e:gQ} we have
\begin{align}
g^*_{k}Q_{\lambda}.1=\sum_{\mbox{\tiny$\begin{array}{c}
\tau_1, \cdots,\tau_{s}\geq0\\ |\tau|=k\end{array}$}}(\prod_{j=1}^{s}f_{\tau_i}(t))Q_{\lambda_1-\tau_1}\cdots Q_{\lambda_{s}-\tau_{s}}.1
\end{align}
where $f_i(t)=q_i(t, -1)$, and $f_i=2(1-t)(i)_t$ for $i>0$ and $f_0=1$.

To compute the $Q_{\mu}$-coefficient $C(\lambda/\mu)$ of $g^*_{k}Q_{\lambda}.1$, we define for $m,n\geq 0$ $$f_{(m,n)}=f_{m}f_{n}+2\sum\limits_{a\geq1}^{n}(-1)^af_{m+a}f_{n-a} \quad (n>0).$$
We make the convention that $f_{(m,0)}=f_m.$ Then $f_{(m,n)}=Q_{(m,n)}(t,-1)$.
\begin{lem}Let $m,n\geq0$, we have\\
(1) If $m>n$, then $f_{(m,n)}=\begin{cases}
2(-t)^{n}f_{m-n}& \text{if $n>0$}\\
f_{m}& \text{if $n=0$}\end{cases}$;\\
(2) $f_{(m,n)}=-f_{(n,m)},$ \quad if $(m,n)\neq(0,0)$. In particular, $f_{(n,n)}=0$, if $n\neq0$
\end{lem}
\begin{proof}
Note that $f_{(m, n)}$ is exactly the value of the Schur Q-function $Q_{(m,n)}(t,-1)$, which implies (2).
By the tableaux presentation of $Q_{(m, n)}$, $Q_{(m,n)}(t,-1)=\sum_{T}2^{b(T)}x^T$ summed over all tableaux satisfying ($n>0$)\\
(i) the first $n$ boxes in the first row are labelled $t$;\\
(ii) all boxes in the second row are labeled $-1$;\\
(iii) the remaining boxes are labeled $t$ or $-1$.\\
Then (1) follows from case by case verification.
%(2) is a direct result from $f_{(m, n)}=Q_{(m,n)}(t,-1)$.  %and $Q_{(m,n)}=-Q_{(n,m)}$.
\end{proof}

Let $\mu\vdash n-k$ be a strict partition, $\mu=(\mu_1,\mu_2,\cdots,\mu_{r})\subset \lambda=(\lambda_1,\lambda_2,\cdots,\lambda_s)$ and $r+s$ is even (where $\mu_r=0$ for $l(\mu)+l(\lambda)$ is odd). Then the coefficient $C(\lambda/\mu)$ of $Q_{\mu}.1$ in $g^*_{k}Q_{\lambda}.1$ can be expressed as follows:
\begin{align*}
C(\lambda/\mu)
=&2^{-l(\mu)}\left \langle Q_{\mu_{1}}Q_{\mu_2}\cdots Q_{\mu_{r}}.1, \sum_{\mbox{\tiny$\begin{array}{c}
i_1,i_2,\cdots,i_{s}\geq0\\i_1+i_2+\cdots+i_{s}=k\end{array}$}}\prod\limits_{j=1}^{s}f_{i_j}Q_{\lambda_1-i_1}\cdots Q_{\lambda_{s}-i_{s}}.1 \right \rangle
\end{align*}

Note that: $\langle Q_{\mu_{1}}\cdots Q_{\mu_{r}}.1, Q_{a_1}\cdots Q_{a_m}.1 \rangle=0$ unless $\sum a_i=|\mu|,$ where $a_i\in \mathbb{Z}.$

Then we have
\begin{align*}
C(\lambda/\mu)=2^{-l(\mu)}\left \langle Q_{\mu_{1}}Q_{\mu_2}\cdots Q_{\mu_{r}}.1, \prod\limits_{j=1}^{s}(\sum\limits_{i_j\geq0}f_{i_j}Q_{\lambda_j-i_j}).1 \right \rangle
\end{align*}

Since $Q_n$ are Fermi operators, it follows from  \eqref{e:wick} that
\begin{align*}
C(\lambda/\mu)=2^{-l(\mu)}(-1)^{n-k}\text{Pf}(M(\lambda/\mu))_{i,j=1,2,\ldots,r+s}
\end{align*}
where $M(\lambda/\mu)$ is a $(r+s)\times(r+s)$ antisymmetric matrix
\begin{gather*}
\begin{pmatrix}M_{r\times r} & M_{r\times s} \\
 -M^{\top}_{s\times r} & M_{s\times s}
\end{pmatrix}_{(r+s)\times (r+s)}
\end{gather*}
where $M_{r\times r}$ and $M_{s\times s}$ are antisymmetric matrices of respective sizes $r$ and $s$ with
\begin{align*}
(M_{r\times r})_{i,j}&=(-1)^{\mu_{r-i+1}}\langle Q_{\mu_{r-i+1}}.1, Q_{-\mu_{r-j+1}}.1\rangle=0, \quad 1\leq i<j\leq r,\\
(M_{r\times s})_{i,j}&=(-1)^{\mu_{r-i+1}}\langle Q_{\mu_{r-i+1}}.1, \sum\limits_{i_j\geq0}f_{i_j}Q_{\lambda_{j}-i_j}.1\rangle\\
&=\begin{cases}
f_{\lambda_{j}}& \text{if $i=1, \mu_{r}=0$}\\
2(-1)^{\mu_{r-i+1}}f_{\lambda_{j}-\mu_{r-i+1}}& \text{if others}
\end{cases}
\quad 1\leq i\leq r, 1\leq j\leq s,\\
(M_{s\times s})_{i,j}&=\left \langle \sum\limits_{i_i\geq0}(-1)^{\lambda_i-i_i}f_{i_i}Q_{i_i-\lambda_i}.1, \sum\limits_{i_j\geq0}f_{i_j}Q_{\lambda_j-i_j}.1\right \rangle\\
&=f_{\lambda_i}f_{\lambda_j}+2\sum\limits_{i_j\geq0}^{\lambda_j-1}(-1)^{\lambda_j-i_j}f_{\lambda_i+\lambda_j-i_j}f_{i_j}\\
&=f_{(\lambda_i,\lambda_j)}, \quad 1\leq i<j\leq s.
\end{align*}

Summarizing the above, we have shown the following result.
\begin{thm}\label{t:Pf}
Let $k$ and $n$ be positive integers $(k\leq n)$, $\lambda\in\mathcal{SP}_n,$ then we have
\begin{align}\label{e:Pf}
g^*_{k}Q_{\lambda}.1=\sum_{\mbox{\tiny$\begin{array}{c}
\mu\subset\lambda\\ \mu\in\mathcal{SP}_{n-k}\end{array}$}}2^{-l(\mu)}(-1)^{n-k}{\rm Pf}(M(\lambda/\mu))Q_{\mu}.1
\end{align}
where $M(\lambda/\mu)$ is defined above.
\end{thm}

We introduce the modified matrix $\tilde{M}(\lambda/\mu)$ of $M(\lambda/\mu)$ as an antisymmetric matrix of size $(r+s)$ with entries
\begin{gather*}
\begin{pmatrix}\tilde{M}_{r\times r} & \tilde{M}_{r\times s} \\
 -\tilde{M}^{\top}_{s\times r} & \tilde{M}_{s\times s}
\end{pmatrix}_{(r+s)\times (r+s)}
\end{gather*}
where $\tilde{M}_{r\times r}$ and $\tilde{M}_{s\times s}$ are antisymmetric matrices given by
\begin{align*}
(\tilde{M}_{r\times r})_{i,j}&=0, \quad 1\leq i<j\leq r,\\
(\tilde{M}_{r\times s})_{i,j}&=f_{\lambda_{j}-\mu_{r-i+1}}, \quad 1\leq i\leq r, 1\leq j\leq s,\\
(\tilde{M}_{s\times s})_{i,j}&=f_{(\lambda_i,\lambda_j)}, \quad 1\leq i<j\leq s.
\end{align*}

By \eqref{e:decomposition1}, we have
\begin{align}\label{e:Pftype}
g^*_{k}Q_{\lambda}.1=\sum_{\mbox{\tiny$\begin{array}{c}
\mu\subset\lambda\\ \mu\in\mathcal{SP}_{n-k}\end{array}$}}{\rm Pf}(\tilde{M}(\lambda/\mu))Q_{\mu}.1
\end{align}

\subsection{Murnaghan-Nakayama type rule for $\zeta^{\la}_{\mu}(q)$.}
We introduce two operations on matrices. Let $B=(b_{ij})_{m\times m}$ be a matrix, we define \\
$(1)$ The column rotation $C_{p}^{q}$ of matrix $B$: columns $p, p+1, \ldots, q-1, q$ are cyclically rotated to columns $q, p, p+1, \ldots, q-1$.\\
$(2)$ The row rotation $R_{p}^q$ of matrix $B$: rows $p, p+1, \ldots, q-1, q$ are cyclically rotated to rows $q, p, p+1, \ldots, q-1$.\\
For instance, suppose $B=(b_{ij})_{4\times 4}$, then
\begin{align*}
C_{1}^3(B)=\left(\begin{array}{cccc}
b_{12}&b_{13}&b_{11}&b_{14}\\
b_{22}&b_{23}&b_{21}&b_{24}\\
b_{32}&b_{33}&b_{31}&b_{34}\\
b_{42}&b_{43}&b_{41}&b_{44}\\
\end{array}\right),~~~
R_{1}^3(B)=\left(\begin{array}{cccc}
b_{21}&b_{22}&b_{23}&b_{24}\\
b_{31}&b_{32}&b_{33}&b_{34}\\
b_{11}&b_{12}&b_{13}&b_{14}\\
b_{41}&b_{42}&b_{43}&b_{44}\\
\end{array}\right).
\end{align*}
Clearly the two operators commute: $C^{i}_{j}R_{p}^{q}=R^{q}_{p}C^{i}_{j}$. If $B$ is an antisymmetric matrix, then $R^{q}_{p}C^{q}_{p}(B)$ is also an antisymmetric matrix and ${\rm Pf}(B)=(-1)^{q-p}{\rm Pf}(R^{q}_{p}C^{q}_{p}(B))$.

Given strict partitions $\la=(\la_1,\la_2,\cdots,\la_s)$, $\mu=(\mu_1,\mu_2,\cdots,\mu_r)$, $s+r$ even. Define the antisymmetric matrix $\hat{M}(\la/\mu)$ as follows:
\begin{align*}
\hat{M}(\la/\mu)=
\left(\begin{array}{cccc}
\hat{M}_{s\times s}& \hat{M}_{s\times r}\\
-\hat{M}^{\top}_{r\times s}& \hat{M}_{r\times r}\\
\end{array}\right)
\end{align*}
where $\hat{M}_{r\times r}$ and $\hat{M}_{s\times s}$ are antisymmetric matrices given by
\begin{align*}
(\hat{M}_{r\times r})_{i,j}&=0, \quad 1\leq i<j\leq r,\\
(\hat{M}_{s\times r})_{i,j}&=f_{\lambda_{i}-\mu_{r-j+1}}, \quad 1\leq i\leq s, 1\leq j\leq r,\\
(\hat{M}_{s\times s})_{i,j}&=f_{(\lambda_i,\lambda_j)}, \quad 1\leq i<j\leq s.
\end{align*}
Using the operators defined above, we can obtain $\hat{M}(\la/\mu)$ from $\tilde{M}(\la/\mu)$. Indeed, we have the following transformations:
\begin{gather*}
  \centering
\begin{tikzpicture}
\node[] (1) at(0,0) {
$
\left(\begin{array}{cccc}
\tilde{M}_{r\times r}& \tilde{M}_{r\times s}\\
-\tilde{M}^{\top}_{s\times r}& \tilde{M}_{s\times s}\\
\end{array}\right)$};
\node[] (2) at(7,0) {$
\left(\begin{array}{cccc}
\tilde{M}_{r\times r}& -\tilde{M}_{r\times s}\\
\tilde{M}^{\top}_{s\times r}& \tilde{M}_{s\times s}\\
\end{array}\right)$};
\node[] (3) at(13.3,0) {$
\left(\begin{array}{cccc}
\hat{M}_{s\times s}& \hat{M}_{s\times r}\\
-\hat{M}^{\top}_{r\times s}& \hat{M}_{r\times r}\\
\end{array}\right).$};
\draw[->] (1)--(2);
\draw[->] (2)--(3);
\node at(3.5,0.2) {\tiny the first $r$ rows times $-1$};
\node at(3.5,-0.2) {\tiny the first $r$ columns times $-1$};
\node at(10.15,0.2) {\tiny $\prod_{i=0}^{r-1}C^{r+s-i}_{r-i}R^{r+s-i}_{r-i}$};
\end{tikzpicture}
\end{gather*}

Therefore, we have
\begin{align}\label{e:tildehat}
{\rm Pf}(\tilde{M}(\la/\mu))=(-1)^{r+rs}{\rm Pf}(\hat{M}(\la/\mu))={\rm Pf}(\hat{M}(\la/\mu))
\end{align}
where the second equation holds for $r+s$ is even.

The skew Schur $Q$-function has a Pfaffian formula given by J\'ozefiak and Pragacz \cite{JP} in terms of Schur $Q$-functions of one row or two rows, i.e.,
\begin{align}
Q_{\la/\mu}(X)={\rm Pf}
\left(\begin{array}{ccccccc}
0& Q_{(\la_1,\la_2)}& \cdots& Q_{(\la_1,\la_s)}& q_{\la_1-\mu_r}& \cdots& q_{\la_1-\mu_1}\\
Q_{(\la_2,\la_1)}& 0& \cdots& Q_{(\la_2,\la_s)}& q_{\la_2-\mu_r}& \cdots& q_{\la_2-\mu_1}\\
\vdots& \vdots& \ddots& \vdots& \vdots& \ddots& \vdots\\
Q_{(\la_s,\la_1)}& Q_{(\la_s,\la_2)}& \cdots& 0& q_{\la_s-\mu_r}& \cdots& q_{\la_s-\mu_1}\\
-q_{\la_1-\mu_r}& -q_{\la_2-\mu_r}& \cdots& -q_{\la_s-\mu_r}& 0& \cdots& 0\\
\vdots& \vdots& \ddots& \vdots& \vdots& \ddots& \vdots\\
-q_{\la_1-\mu_1}& -q_{\la_2-\mu_1}& \cdots& -q_{\la_s-\mu_1}& 0& \cdots& 0
\end{array}\right).
\end{align}
This is an analog to the similar result of Lascoux and Pragacz for ordinary skew Schur functions \cite{LP}. Since $f_i=q_{i}(t,-1)$ and $f_{(n,m)}=Q_{(n,m)}(t,-1)$, we have
\begin{align}\label{e:QPf}
Q_{\la/\mu}(t,-1)={\rm Pf}(\hat{M}(\la/\mu)).
\end{align}

The following result is a consequence of \eqref{e:tildehat} and \eqref{e:Pftype}.

\begin{thm} Let $k\leq n$ be two positive integers. Then for any $\lambda\in\mathcal{SP}_n,$ we have
\begin{align}
g^*_{k}Q_{\la}.1=\sum_{\mbox{\tiny$\begin{array}{c}
\mu\subset\lambda\\ \mu\in\mathcal{SP}_{n-k}\end{array}$}}Q_{\la/\mu}(t,-1)Q_{\mu}.1.
\end{align}
\end{thm}

Given two strict partitions $\mu\subset\la$. We call $\la/\mu$ is a {\it generalized strip} if the shifted skew diagram $\la^*/\mu^*$ has no $2\times 2$ block of squares. Particularly a connected generalized strip is exactly a {\it shifted border strip}. $\la/\mu$ is called a {\it double strip} if it is formed by the union of two shifted border strips which both end on the main diagonal. A double strip can be cut into two non-empty connected pieces, one piece consisting of the diagonals of length $2$, and other piece consisting of the shifted border strip formed by the diagonals of length $1$. The following is an example for $\la=(6,5,4,3,2)$ and $\mu=(5,4,2)$. The boxes occupied by two red lines form respectively two shifted border strips. The piece of length $2$ (resp. $1$) is composed of all boxes marked by $\ast$'s (resp. $\bullet$'s).
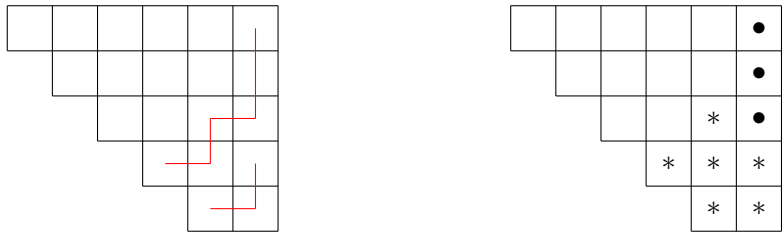
\begin{gather*}
  \centering
\begin{tikzpicture}[scale=0.6]
   \coordinate (Origin)   at (0,0);
    \coordinate (XAxisMin) at (0,0);
    \coordinate (XAxisMax) at (6,0);
    \coordinate (YAxisMin) at (0,5);
    \coordinate (YAxisMax) at (0,0);
    \draw [thin, black] (0,4) -- (0,5);
    \draw [thin, black] (1,3) -- (1,5);
    \draw [thin, black] (2,2) -- (2,5);
    \draw [thin, black] (3,1) -- (3,5);
    \draw [thin, black] (4,0) -- (4,5);
    \draw [thin, black] (5,0) -- (5,5);
    \draw [thin, black] (6,0) -- (6,5);
    \draw [thin, black] (4,0) -- (6,0);
    \draw [thin, black] (3,1) -- (6,1);
    \draw [thin, black] (2,2) -- (6,2);
    \draw [thin, black] (1,3) -- (6,3);
    \draw [thin, black] (0,4) -- (6,4);
    \draw [thin, black] (0,5) -- (6,5);
    \draw [thin, red] (3.5,1.5) -- (4.5,1.5);
    \draw [thin, red] (4.5,1.5) -- (4.5,2.5);
    \draw [thin, red] (4.5,2.5) -- (5.5,2.5);
    \draw [thin, red] (5.5,2.5) -- (5.5,4.5);
    \draw [thin, red] (4.5,0.5) -- (5.5,0.5);
    \draw [thin, red] (5.5,0.5) -- (5.5,1.5);
\end{tikzpicture}\qquad\qquad\qquad\qquad
\begin{tikzpicture}[scale=0.6]
   \coordinate (Origin)   at (0,0);
    \coordinate (XAxisMin) at (0,0);
    \coordinate (XAxisMax) at (6,0);
    \coordinate (YAxisMin) at (0,5);
    \coordinate (YAxisMax) at (0,0);
    \draw [thin, black] (0,4) -- (0,5);
    \draw [thin, black] (1,3) -- (1,5);
    \draw [thin, black] (2,2) -- (2,5);
    \draw [thin, black] (3,1) -- (3,5);
    \draw [thin, black] (4,0) -- (4,5);
    \draw [thin, black] (5,0) -- (5,5);
    \draw [thin, black] (6,0) -- (6,5);
    \draw [thin, black] (4,0) -- (6,0);
    \draw [thin, black] (3,1) -- (6,1);
    \draw [thin, black] (2,2) -- (6,2);
    \draw [thin, black] (1,3) -- (6,3);
    \draw [thin, black] (0,4) -- (6,4);
    \draw [thin, black] (0,5) -- (6,5);
    \node[inner sep=2pt] at (3.5,1.5) {$\ast$};
    \node[inner sep=2pt] at (4.5,1.5) {$\ast$};
    \node[inner sep=2pt] at (4.5,2.5) {$\ast$};
    \node[inner sep=2pt] at (5.5,1.5) {$\ast$};
    \node[inner sep=2pt] at (5.5,0.5) {$\ast$};
    \node[inner sep=2pt] at (4.5,0.5) {$\ast$};
    \node[inner sep=2pt] at (5.5,2.5) {$\bullet$};
    \node[inner sep=2pt] at (5.5,3.5) {$\bullet$};
    \node[inner sep=2pt] at (5.5,4.5) {$\bullet$};
\end{tikzpicture}
\end{gather*}
More generally, we call $\la/\mu$ a {\it generalized double strip} if there exists a strict partition $\mu\subset\nu\subset\la$ such that $\la^*/\nu^*$ and $\nu^*/\mu^*$ are both generalized strips. For a nonnegative integer $k$, a generalized double strip $\la/\mu$ is called a $k$-generalized double strip if $|\la/\mu|=k$.
\begin{exmp}\label{t:GDS}
$\la=(15,14,10,8,7,6,5,3,1)$, $\mu=(13,11,8,6,5,4,2,1)$ then $\la/\mu$ is a generalized double strip. Indeed we can choose $\nu=(15,13,10,8,6,5,4,2,1)$. The shifted diagram can be drawn as follows.
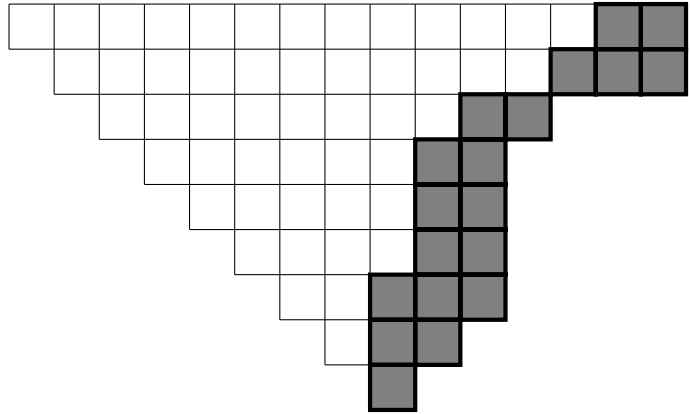
\begin{gather}
  \centering
\begin{tikzpicture}[scale=0.6]
   \coordinate (Origin)   at (0,0);
    \coordinate (XAxisMin) at (0,0);
    \coordinate (XAxisMax) at (15,0);
    \coordinate (YAxisMin) at (0,0);
    \coordinate (YAxisMax) at (0,9);
    \draw [thin, black] (0,8) -- (0,9);
    \draw [thin, black] (1,7) -- (1,9);
    \draw [thin, black] (2,6) -- (2,9);
    \draw [thin, black] (3,5) -- (3,9);
    \draw [thin, black] (4,4) -- (4,9);
    \draw [thin, black] (5,3) -- (5,9);
    \draw [thin, black] (6,2) -- (6,9);
    \draw [thin, black] (7,1) -- (7,9);
    \draw [thin, black] (8,0) -- (8,9);
    \draw [thin, black] (9,0) -- (9,9);
    \draw [thin, black] (10,1) -- (10,9);
    \draw [thin, black] (11,2) -- (11,9);
    \draw [thin, black] (12,6) -- (12,9);
    \draw [thin, black] (13,7) -- (13,9);
    \draw [thin, black] (14,7) -- (14,9);
    \draw [thin, black] (15,7) -- (15,9);
    \draw [thin, black] (8,0) -- (9,0);
    \draw [thin, black] (7,1) -- (10,1);
    \draw [thin, black] (6,2) -- (11,2);
    \draw [thin, black] (5,3) -- (11,3);
    \draw [thin, black] (4,4) -- (11,4);
    \draw [thin, black] (3,5) -- (11,5);
    \draw [thin, black] (2,6) -- (12,6);
    \draw [thin, black] (1,7) -- (15,7);
    \draw [thin, black] (0,8) -- (15,8);
    \draw [thin, black] (0,9) -- (15,9);
    \filldraw[fill = gray][ultra thick]
    (8,1) rectangle (9,0) (8,2)rectangle(9,1) (8,2)rectangle(9,3) (9,1)rectangle(10,2) (9,2)rectangle(10,3) (9,3)rectangle(10,4) (9,4)rectangle(10,5) (9,5)rectangle(10,6) (10,2)rectangle(11,3) (10,3)rectangle(11,4) (10,4)rectangle(11,5) (10,5)rectangle(11,6) (10,6)rectangle(11,7) (11,6)rectangle(12,7) (12,7)rectangle(13,8) (13,7)rectangle(14,8) (13,8)rectangle(14,9) (14,7)rectangle(15,8) (14,8)rectangle(15,9);
\end{tikzpicture}
\end{gather}
\end{exmp}
Similar to a double strip, a generalized double strip can be cut into two kinds of pieces (may be empty and not necessarily connected). One kind of pieces (denoted by $\alpha(\la/\mu)$) consist of the diagonals of length $2$, and other kind (denoted by $\beta(\la/\mu)$) consisting of the generalized strips formed by the diagonals of length $1$. In the first kind of pieces, we call the box on the top left (resp. bottom right) $t$-box (resp. $-1$-box). We label the $t$-boxes (resp. $-1$-boxes) with $t$ (resp. $-1$) in $\la^*/\mu^*$. Clearly, the number of $t$-boxes and $-1$-boxes in $\alpha(\la/\mu)$ are same. Denote this number by $c(\la/\mu)$. $\beta(\la/\mu)$ is a union of its connected components, each of which is a shifted border strip. Denote the number of connected components in $\beta(\la/\mu)$ by $m(\la/\mu)$. The following graph is for Example \ref{t:GDS} and $c(\la/\mu)=5$, $m(\la/\mu)=5$.
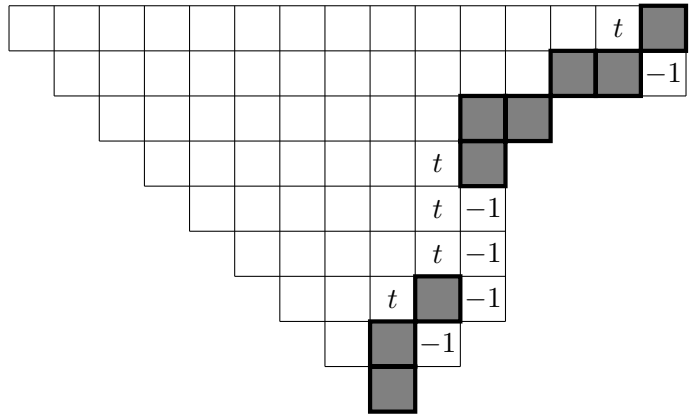
\begin{gather}
  \centering
\begin{tikzpicture}[scale=0.6]
   \coordinate (Origin)   at (0,0);
    \coordinate (XAxisMin) at (0,0);
    \coordinate (XAxisMax) at (15,0);
    \coordinate (YAxisMin) at (0,0);
    \coordinate (YAxisMax) at (0,9);
    \draw [thin, black] (0,8) -- (0,9);
    \draw [thin, black] (1,7) -- (1,9);
    \draw [thin, black] (2,6) -- (2,9);
    \draw [thin, black] (3,5) -- (3,9);
    \draw [thin, black] (4,4) -- (4,9);
    \draw [thin, black] (5,3) -- (5,9);
    \draw [thin, black] (6,2) -- (6,9);
    \draw [thin, black] (7,1) -- (7,9);
    \draw [thin, black] (8,0) -- (8,9);
    \draw [thin, black] (9,0) -- (9,9);
    \draw [thin, black] (10,1) -- (10,9);
    \draw [thin, black] (11,2) -- (11,9);
    \draw [thin, black] (12,6) -- (12,9);
    \draw [thin, black] (13,7) -- (13,9);
    \draw [thin, black] (14,7) -- (14,9);
    \draw [thin, black] (15,7) -- (15,9);
    \draw [thin, black] (8,0) -- (9,0);
    \draw [thin, black] (7,1) -- (10,1);
    \draw [thin, black] (6,2) -- (11,2);
    \draw [thin, black] (5,3) -- (11,3);
    \draw [thin, black] (4,4) -- (11,4);
    \draw [thin, black] (3,5) -- (11,5);
    \draw [thin, black] (2,6) -- (12,6);
    \draw [thin, black] (1,7) -- (15,7);
    \draw [thin, black] (0,8) -- (15,8);
    \draw [thin, black] (0,9) -- (15,9);
    \node[inner sep=2pt] at (8.5,2.5) {$t$};
    \node[inner sep=2pt] at (9.5,3.5) {$t$};
    \node[inner sep=2pt] at (9.5,4.5) {$t$};
    \node[inner sep=2pt] at (9.5,5.5) {$t$};
    \node[inner sep=2pt] at (13.5,8.5) {$t$};
    \node[inner sep=2pt] at (9.5,1.5) {$-1$};
    \node[inner sep=2pt] at (10.5,2.5) {$-1$};
    \node[inner sep=2pt] at (10.5,3.5) {$-1$};
    \node[inner sep=2pt] at (10.5,4.5) {$-1$};
    \node[inner sep=2pt] at (14.5,7.5) {$-1$};
    \filldraw[fill = gray][ultra thick]
    (8,1) rectangle (9,0) (8,2)rectangle(9,1)   (9,2)rectangle(10,3)  (10,5)rectangle(11,6) (10,6)rectangle(11,7) (11,6)rectangle(12,7) (12,7)rectangle(13,8) (13,7)rectangle(14,8)  (14,8)rectangle(15,9);
\end{tikzpicture}
\end{gather}

We remark that a generalized strip (or double strip) is naturally a generalized double strip. Let $\la/\mu$ be a generalized double strip, we have (1) if $c(\la/\mu)\geq1$, $m(\la/\mu)=1$, then $\la/\mu$ is a double strip; (2) if $c(\la/\mu)=0$, then $\la/\mu$ is a generalized strip; (3) if $c(\la/\mu)=0$, $m(\la/\mu)=1$, then $\la/\mu$ is a shifted border strip.

Let $\tau=(\tau_1,\tau_2,\cdots,\tau_a),\rho=(\rho_1,\rho_2,\cdots,\rho_b)$ be two compositions. We call $\rho$ is a {\it refinement} of $\tau$, denoted $\rho\prec\tau$, if there exists $i_0=0< i_1< i_2<\cdots< i_{a-1}< b=i_{a}$ such that $\tau_j=\rho_{i_{j-1}+1}+\rho_{i_{j-1}+2}+\cdots+\rho_{i_j}$, $j=1,2,\cdots,a$. For convenience, we also say that $\tau$ is a {\it coarsening} composition of $\rho$. %In this case, we denote .
In particular, $\tau\prec\tau$. Note that a given composition has only finitely many coarsening ones. %there are finitely many refinements

Now we can give the Murnaghan-Nakayama rule. Let us begin by recalling the combinatorial definition of skew Schur $Q$-functions.
\begin{align}
Q_{\la/\mu}(X)=\sum_{T}2^{b(T)}x^{\alpha}
\end{align}
summed over all shifted tableaux of shape $\theta^*=\la^*/\mu^*$.

We remark that the diagonal square $(i,i)$ in $\theta^*$ must all lie in distinct shifted border strips, so that $b(T)\geq l(\la)-l(\mu)$ for each tableau $T$ of shape $\theta^*$. Thus $Q_{\la/\mu}(x_1,x_2,\ldots,x_n)=0$ unless $n\geq l(\la)-l(\mu)$. It is clear that if $\theta^*$ has $m$ connected components $\xi_1^*,\xi_2^*,\cdots,\xi_m^*$, then $Q_{\theta}(X)=\prod_{i=1}^{m}Q_{\xi_i}(X)$.

By the above definition, in the case where there is only one variable $x$, we have $Q_{\la/\mu}(x)=0$ unless $\la/\mu$ is a generalized strip. Note that
\begin{align}
Q_{\la/\mu}(x_1,x_2)=\sum_{\nu}Q_{\nu/\mu}(x_1)Q_{\la/\nu}(x_2)
\end{align}
where the sum runs over all strict partition $\nu$ such that $\mu\subset\nu\subset\la$. Therefore $Q_{\la/\mu}(x_1,x_2)=0$ unless $\la/\mu$ is a generalized double strip.

Let $\la/\mu$ be a generalized double strip, we can compute $Q_{\la/\mu}(t,-1)$ by shifted tableaux of shape $\la^*/\mu^*$ labelled $t$ or $-1$. Indeed we have
\begin{align}\label{e:Q(t,-1)}
Q_{\la/\mu}(t,-1)=\sum_{T}2^{b(T)}t^{\alpha_1}(-1)^{|\la/\mu|-\alpha_1}
\end{align}
where the sum runs over all shifted tableaux of shape $\la^*/\mu^*$ satisfying\\
(M1) the $t$-boxes (resp. $-1$-boxes) are labeled $t$ (resp. $-1$);\\
(M2) $-1$'s are behind (resp. below) $t$'s in each row (resp. column).\\
And $\alpha_1$ is the number of boxes in $T$ labeled $t$.

Given a generalized double strip $\la/\mu$, $\beta(\la/\mu)$ has $m$ connected components denoted by $\xi^{(1)},\cdots,\xi^{(m)}$, we define its weight by
\begin{align}
wt_t(\la/\mu)=(-t)^{c(\la/\mu)}2^{\delta_{l(\la),l(\mu)+2}}\prod_{i=1}^{m}\left(\sum_{\xi^{(i)}\prec\tau}(-1)^{l(\xi^{(i)})-l(\tau)}f_{\tau}\right)
\end{align}
where $f_{\tau}=f_{\tau_1}f_{\tau_2}\cdots$.

\begin{exmp}
In example \ref{t:GDS}, $\la=(15,14,10,8,7,6,5,3,1)$ and $\mu=(13,11,8,6,5,4,2,1)$, $l(\la)=l(\mu)+1$. So
\begin{align*}
wt_t(\la/\mu)=(-t)^5f_1f_2(f_1f_2-f_3)f_1(f_1f_1-f_2)=-2^5t^5(t-1)^7(t^2-3t+1).
\end{align*}
\end{exmp}

Now we want to compute $Q_{\la/\mu}(t,-1)$ and show $Q_{\la/\mu}(t,-1)=wt_t(\la/\mu)$ when $\la/\mu$ is a generalized double strip. Let us firstly look at the simplest case: $\la/\mu$ is a shifted border strip.
\begin{prop}\label{t:SBS}
Suppose $\la/\mu$ ($l(\la/\mu)=s$) is a shifted border strip and $\la/\mu$ has $r_i$ boxes in the $i$th row, $1\leq i\leq s$. Then we have
\begin{align}\label{e:SBS}
Q_{\la/\mu}(t,-1)=\sum_{\tau}(-1)^{s-l(\tau)}f_{\tau}
\end{align}
summed over all coarsening compositions $\tau$ of $(r_1,\cdots,r_s)$, i.e, $(r_1,\cdots,r_s)\prec\tau$.
\end{prop}
\begin{proof}
By \eqref{e:Q(t,-1)}, without loss of generality, we can assume that $\la=(\la_1,\la_2,\cdots,\la_s)$ and $\mu=(\la_2,\la_3,\cdots,\la_{s},0)$. Then $\la_i-\la_{i+1}=r_i$ ($i=1,2,\cdots,s-1$) and $\la_s=r_s$. By \eqref{e:QPf}, we have
\begin{align*}
Q_{\la/\mu}(t,-1)={\rm Pf}(\hat{M}(\la/\mu))&=(-1)^{\frac{s(s-1)}{2}}\det
\left(\begin{array}{cccc}
f_{\la_1}& f_{\la_1-\la_s}& \cdots& f_{\la_1-\la_2}\\
f_{\la_2}& f_{\la_2-\la_s}& \cdots& 1\\
\vdots& \vdots& \begin{rotate}{90}$\ddots$\end{rotate}& \vdots\\
f_{\la_s}& 1& 0& 0
\end{array}\right)\\
&=\det
\left(\begin{array}{cccc}
f_{\la_1-\la_2}& f_{\la_1-\la_3}& \cdots& f_{\la_1}\\
1& f_{\la_2-\la_3}& \cdots& f_{\la_2}\\
\vdots& \ddots& \ddots& \vdots\\
0& \cdots& 1&  f_{\la_s}
\end{array}\right)\\
&=\sum_{(r_1,\cdots,r_s)\prec\tau}(-1)^{s-l(\tau)}f_{\tau}.
\end{align*}
The last equation holds by induction on $s$.
\end{proof}

\begin{rem}\label{r:sbs}
When $\la/\mu$ is a shifted border strip, by Proposition \ref{t:SBS} we have $(t-1)\mid Q_{\la/\mu}(t,-1)$ and
\begin{align*}
\frac{Q_{\la/\mu}(t,-1)}{t-1}\mid_{t=1}=(-1)^{s-1}\frac{f_{r_1+\cdots+r_s}}{t-1}\mid_{t=1}&=(-1)^{s-1}2\frac{1-(-1)^{|\la/\mu|}}{2}\\
&=\begin{cases}
2(-1)^{ht(\la/\mu)}& \text{if $|\la/\mu|$ is odd};\\
0& \text{if $|\la/\mu|$ is even}.
\end{cases}
\end{align*}
\end{rem}

We now compute $Q_{\la/\mu}(t,-1)$ when $\la/\mu$ is a generalized double strip.

\begin{thm}\label{t:casetwo}
Let $\la/\mu$ be a generalized double strip. Suppose $\beta(\la/\mu)$ has two connected components (shifted border strips) denoted by $\xi_1,\xi_2$. Then we have
\begin{align}
Q_{\la/\mu}(t,-1)=(-t)^{c(\la/\mu)}2^{\delta_{l(\la),l(\mu)+2}}Q_{\xi_1}(t,-1)Q_{\xi_2}(t,-1)
\end{align}
where $\delta$ is the Kronecker delta function.
\end{thm}
\begin{proof}
Since $\la/\mu$ is a generalized double strip, $l(\mu)\leq l(\la)\leq l(\mu)+2$. We divide it into two cases.\\
{\bf Case 1:} $l(\mu)\leq l(\la)\leq l(\mu)+1$. In this case, $\alpha(\la/\mu)$ is connected. The skew shifted diagram of $\la/\mu$ can be drawn as follows. Here $a,b\in\{t,-1\}$.
\begin{gather*}
  \centering
\begin{tikzpicture}[scale=0.6]
   \coordinate (Origin)   at (0,0);
    \coordinate (XAxisMin) at (0,0);
    \coordinate (XAxisMax) at (15,0);
    \coordinate (YAxisMin) at (0,0);
    \coordinate (YAxisMax) at (0,9);
    \draw [thin, black] (0,0) -- (0,2);
    \draw [thin, black] (1,0) -- (1,2);
    \draw [thin, black] (2,0) -- (2,1);
    \draw [thin, black] (0,0) -- (2,0);
    \draw [thin, black] (0,1) -- (2,1);
    \draw [thin, black] (0,2) -- (1,2);
    \node[inner sep=2pt] at (0.5,1.5) {$t$};
    \node[inner sep=2pt] at (1.5,0.5) {$-1$};
    \node[inner sep=2pt] at (3,1.5) {$\begin{rotate}{90}$\ddots$\end{rotate}$};
    \node[inner sep=2pt] at (2,2.5) {$\begin{rotate}{90}$\ddots$\end{rotate}$};
    \draw [thin, black] (3,4) -- (3,5);
    \draw [thin, black] (4,3) -- (4,5);
    \draw [thin, black] (5,3) -- (5,5);
    \draw [thin, black] (4,3) -- (5,3);
    \draw [thin, black] (3,4) -- (5,4);
    \draw [thin, black] (3,5) -- (5,5);
    \node[inner sep=2pt] at (3.5,4.5) {$t$};
    \node[inner sep=2pt] at (4.5,3.5) {$-1$};
    \node[inner sep=2pt] at (0.5,0.5) {$b$};
    \node[inner sep=2pt] at (4.5,4.5) {$a$};
    \node[inner sep=2pt] at (-0.5,-1) {\begin{rotate}{90}$\ddots$\end{rotate}};
    \node[inner sep=2pt] at (6,6) {\begin{rotate}{90}$\ddots$\end{rotate}};
    \node[inner sep=2pt] at (-0.2,-1) {$\xi_2$};
    \node[inner sep=2pt] at (6,5.5) {$\xi_1$};
\end{tikzpicture}
\end{gather*}
So $(a,b)=(t,-1),(-1,t),(t,t),(-1,-1)$. $(\romannumeral1)$ When $(a,b)=(t,-1)$. For any shifted tableaux $T$ of shape $\la^*/\mu^*$ satisfying (M1) and (M2), we can divide $T$ into two disjoint parts $\gamma_1$ and $\gamma_2$. Here $\gamma_1=\xi_1\cup\{t\text{-boxes}\}$ and $\gamma_2=\xi_2\cup\{-1\text{-boxes}\}$. Then $b(T)=b(\gamma_1)+b(\gamma_2)=b(\xi_1)+b(\xi_2)$. $(\romannumeral2)$ When $(a,b)=(t,-1)$. It is similar to $(\romannumeral1)$. $(\romannumeral3)$ When $(a,b)=(t,t)$. For any shifted tableaux $T$ of shape $\la^*/\mu^*$ satisfying (M1) and (M2), we can also divide $T$ into two disjoint parts $\gamma_{1}'$ and $\gamma_{2}'$. And $\gamma_{1}'=\xi_1\cup\{t\text{-boxes}\}\cup\xi_2$, $\gamma_{2}'=\{-1\text{-boxes}\}$. Then $b(\gamma_{1}')=b(\xi_1)+b(\xi_2)-1$ and $b(\gamma_{2}')=1$. So $b(T)=b(\gamma_{1}')+b(\gamma_{2}')=b(\xi_1)+b(\xi_2)$. $(\romannumeral4)$ When $(a,b)=(-1,-1)$. It is similar to $(\romannumeral3)$. So we always have $b(T)=b(\xi_1)+b(\xi_2)$. Therefore $Q_{\la/\mu}(t,-1)=(-t)^{c(\la/\mu)}Q_{\xi_1}(t,-1)Q_{\xi_2}(t,-1)$.\\
{\bf Case 2:} $l(\la)=l(\mu)+2$. In this case $\alpha(\la/\mu)$ has two connected components and the boxes in the main diagonal belong to $\alpha(\la/\mu)$. The skew shifted diagram of $\la/\mu$ can be drawn as follows. Here $a,b,c\in\{t,-1\}$.
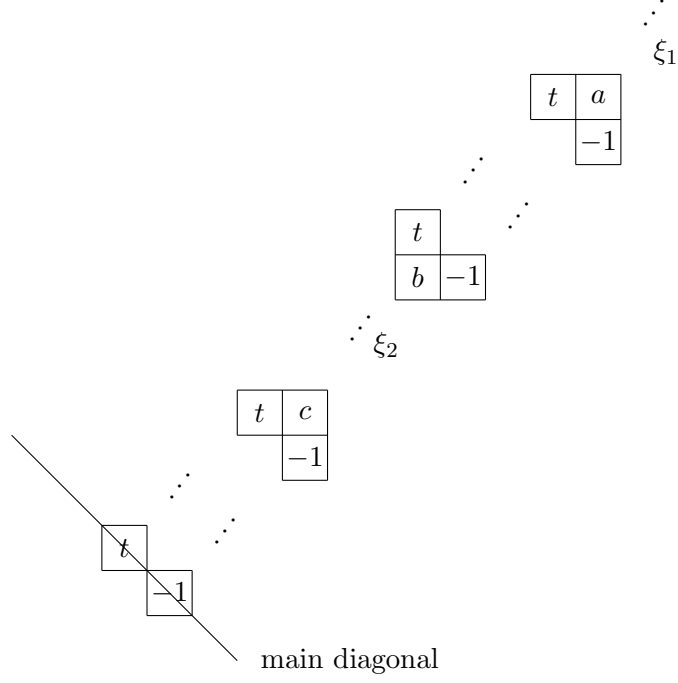
\begin{gather*}
  \centering
\begin{tikzpicture}[scale=0.6]
   \coordinate (Origin)   at (0,0);
    \coordinate (XAxisMin) at (0,0);
    \coordinate (XAxisMax) at (15,0);
    \coordinate (YAxisMin) at (0,0);
    \coordinate (YAxisMax) at (0,9);
    \draw [thin, black] (0,0) -- (0,2);
    \draw [thin, black] (1,0) -- (1,2);
    \draw [thin, black] (2,0) -- (2,1);
    \draw [thin, black] (0,0) -- (2,0);
    \draw [thin, black] (0,1) -- (2,1);
    \draw [thin, black] (0,2) -- (1,2);
    \node[inner sep=2pt] at (0.5,1.5) {$t$};
    \node[inner sep=2pt] at (1.5,0.5) {$-1$};
    \node[inner sep=2pt] at (3,1.5) {$\begin{rotate}{90}$\ddots$\end{rotate}$};
    \node[inner sep=2pt] at (2,2.5) {$\begin{rotate}{90}$\ddots$\end{rotate}$};
    \draw [thin, black] (3,4) -- (3,5);
    \draw [thin, black] (4,3) -- (4,5);
    \draw [thin, black] (5,3) -- (5,5);
    \draw [thin, black] (4,3) -- (5,3);
    \draw [thin, black] (3,4) -- (5,4);
    \draw [thin, black] (3,5) -- (5,5);
    \node[inner sep=2pt] at (3.5,4.5) {$t$};
    \node[inner sep=2pt] at (4.5,3.5) {$-1$};
    \node[inner sep=2pt] at (0.5,0.5) {$b$};
    \node[inner sep=2pt] at (4.5,4.5) {$a$};
    \node[inner sep=2pt] at (-0.5,-1) {\begin{rotate}{90}$\ddots$\end{rotate}};
    \node[inner sep=2pt] at (6,6) {\begin{rotate}{90}$\ddots$\end{rotate}};
    \node[inner sep=2pt] at (-0.2,-1) {$\xi_2$};
    \node[inner sep=2pt] at (6,5.5) {$\xi_1$};
    \draw [thin, black] (-1.5,-2) -- (-1.5,-4);
    \draw [thin, black] (-1.5,-2) -- (-3.5,-2);
    \draw [thin, black] (-2.5,-2) -- (-2.5,-4);
    \draw [thin, black] (-3.5,-2) -- (-3.5,-3);
    \draw [thin, black] (-3.5,-3) -- (-1.5,-3);
    \draw [thin, black] (-2.5,-4) -- (-1.5,-4);
    \node[inner sep=2pt] at (-4.5,-4.5) {\begin{rotate}{90}$\ddots$\end{rotate}};
    \node[inner sep=2pt] at (-3.5,-5.5) {\begin{rotate}{90}$\ddots$\end{rotate}};
    \node[inner sep=2pt] at (-3,-2.5) {$t$};
    \node[inner sep=2pt] at (-2,-3.5) {$-1$};
    \node[inner sep=2pt] at (-2,-2.5) {$c$};
    \draw [thin, black] (-6.5,-5) -- (-6.5,-6);
    \draw [thin, black] (-5.5,-5) -- (-5.5,-7);
    \draw [thin, black] (-4.5,-6) -- (-4.5,-7);
    \draw [thin, black] (-6.5,-5) -- (-5.5,-5);
    \draw [thin, black] (-6.5,-6) -- (-4.5,-6);
    \draw [thin, black] (-5.5,-7) -- (-4.5,-7);
    \node[inner sep=2pt] at (-6,-5.5) {$t$};
    \node[inner sep=2pt] at (-5,-6.5) {$-1$};
    \draw [thin, black] (-8.5,-3) -- (-3.5,-8);
    \node[inner sep=2pt] at (-1,-8) {main diagonal};
\end{tikzpicture}
\end{gather*}
Similar to {\bf Case 1}, we discuss it case by case. For any shifted tableaux $T$ of shape $\la^*/\mu^*$ satisfying (M1) and (M2), we always have $b(T)=b(\xi_1)+b(\xi_2)+1$. Therefore we have $Q_{\la/\mu}(t,-1)=(-t)^{c(\la/\mu)}2Q_{\xi_1}(t,-1)Q_{\xi_2}(t,-1)$.
\end{proof}

More generally, we have:
\begin{cor}\label{c:gds}
Let $\la/\mu$ be a generalized double strip, $\beta(\la/\mu)$ has $m$ connected components denoted by $\xi^{(1)},\cdots,\xi^{(m)}$, then we have
\begin{align}\label{e:gds}
Q_{\la/\mu}(t;-1)=(-t)^{c(\la/\mu)}2^{\delta_{l(\la),l(\mu)+2}}\prod_{i=1}^{m}\left(\sum_{\xi^{(i)}\prec\tau}(-1)^{l(\xi^{(i)})-l(\tau)}f_{\tau}\right)
=wt_t(\la/\mu).
\end{align}
\end{cor}
\begin{proof}
It follows from Proposition \ref{t:SBS} and Theorem \ref{t:casetwo}.
\end{proof}

The Murnaghan-Nakayama rule can be stated as follows:
\begin{cor}\label{c:comb-MN}
Let $\la\in\mathcal{SP}_n$ and $\mu\in\mathcal{OP}_n$, then
\begin{align}
\zeta^{\la}_{\mu}(q)=\sum_{\nu}2^{\epsilon(\nu)-\epsilon(\la)}(q-1)^{-1}wt_q(\la/\nu)\zeta^{\nu}_{\mu^{[1]}}(q)
\end{align}
where the sum runs over all strict partitions $\nu\subset\la$ such that $\la/\nu$ is a $\mu_1$-generalized double strip.
\end{cor}

\begin{rem}
Let $\la/\mu$ be a generalized double strip. It follows from Remark \ref{r:sbs} and Corollary \ref{c:gds} that $\frac{Q_{\la/\mu}(t;-1)}{t-1}\mid_{t=1}=0$ unless $(\romannumeral1)$ $c(\la/\mu)=0$, $m(\la/\mu)=1$ or $(\romannumeral2)$ $c(\la/\mu)\geq 1$, $m(\la/\mu)=1$, i.e., $\la/\mu$ is a (\romannumeral1) shifted border strip or (\romannumeral2) double strip.\\
{\bf Case 1:} When $\la/\mu$ is a shifted border strip, then $l(\mu)\leq l(\la)\leq l(\mu)+1$. By Remark \ref{r:sbs} and Corollary \ref{c:gds}, we have
\begin{align*}
\frac{Q_{\la/\mu}(t,-1)}{t-1}\mid_{t=1}=
\begin{cases}
2(-1)^{ht(\la/\mu)}& \text{if $|\la/\mu|$ is odd};\\
0& \text{if $|\la/\mu|$ is even}.
\end{cases}
\end{align*}
{\bf Case 2:} When $\la/\mu$ is a shifted border strip, then $l(\la)=l(\mu)+2$. By Corollary \ref{c:gds}, we have
\begin{align*}
\frac{Q_{\la/\mu}(t,-1)}{t-1}\mid_{t=1}=
\begin{cases}
4(-1)^{c(\la/\mu)+ht(\beta(\la/\mu))}& \text{if $|\la/\mu|$ is odd};\\
0& \text{if $|\la/\mu|$ is even}.
\end{cases}
\end{align*}
Denote
\begin{align*}
\pi(\la/\mu)=
\begin{cases}
(-1)^{ht(\la/\mu)}& \text{if $\la/\mu$ is a shifted border strip};\\
2(-1)^{c(\la/\mu)+ht(\beta(\la/\mu))}& \text{if $\la/\mu$ is a double strip}.
\end{cases}
\end{align*}
Note that
\begin{align*}
\tilde{g}_{k}(t)\mid_{t=1}&=
\begin{cases}
2p_k& \text{if $k$ is odd};\\
0& \text{if $k$ is even}.
\end{cases}\\
\tilde{g}_{k}(t)&=\frac{1}{t-1}g_k.
\end{align*}
Therefore
\begin{align*}
p^*_{k}Q_{\la}=\sum\pi(\la/\mu)Q_{\mu} \quad (k \text{ is odd})
\end{align*}
summed over all strict partitions $\mu\subset\la$ such that $\la/\mu$ is either a $k$-shifted border strip (case \romannumeral1) or a $k$-double strip (case \romannumeral2).

The Hecke-Clifford algebra $\mathcal{H}^c_n$ $q$-deforms ($q=1$) the algebra $\mathfrak{H}^c_n=C_n\rtimes S_n$. So our result includes the classical Murnaghan-Nakayama rule for algebra $\mathfrak{H}_n^c=C_n\rtimes S_n$ \cite[p. 266]{Mac} and also explains why generalized double strips are not needed in the case of $\mathfrak{H}_n^c$. (There seems a typo in formula (8) in \cite[p. 266]{Mac}: a factor of 2 is missing).
\end{rem}

Recall that a partition-valued function $\underline{\lambda}=(\lambda^{(1)},\lambda^{(2)},\cdots)$ is a sequence of partitions. If all $\lambda^{(i)}$ are strict, $\underline{\lambda}$ is called a strict partition-valued function.
It is called {\it ordered} if $\lambda^{(i)}\supset\lambda^{(i+1)}$ for any $i\geq 1$. For an ordered strict partition-value function $\underline{\lambda}=(\lambda^{(1)},\lambda^{(2)},\cdots,\lambda^{(r)}),$ we define its weight by
\begin{align}
wt(\underline{\lambda})\doteq \prod_{i=1}^{r-1} wt_q(\lambda^{(i)}/\lambda^{(i+1)}).
\end{align}
The following is immediate.
\begin{cor} \label{t:general}
Let $\lambda\in\mathcal{SP}_n$, $\mu\in\mathcal{OP}_n,$ $l(\mu)=m.$ Then we have
\begin{align}\label{e:general}
\zeta^{\lambda}_{\mu}(q)=2^{-\epsilon(\lambda)}(q-1)^{-m}\sum_{\underline{\lambda}}wt(\underline{\lambda})
\end{align}
summed over all ordered strict partition-value functions $\underline{\lambda}=(\lambda^{(0)},\lambda^{(1)},\lambda^{(2)},\cdots,\lambda^{(m)})$ ($\la^{(0)}=\la$) such that $\lambda^{(i-1)}/\la^{(i)}$ is a $\mu_i$-generalized double strip $i=1,2,\ldots,m.$
\end{cor}

\begin{exmp}
Given $\lambda=(4,2,1),$ $\mu=(3,3,1),$ then $\lambda^{(1)}=(4)~ {\rm or}~(3,1),$ $\lambda^{(2)}=(1),$ $\lambda^{(3)}=\emptyset.$ The relevant values are calculated as follows
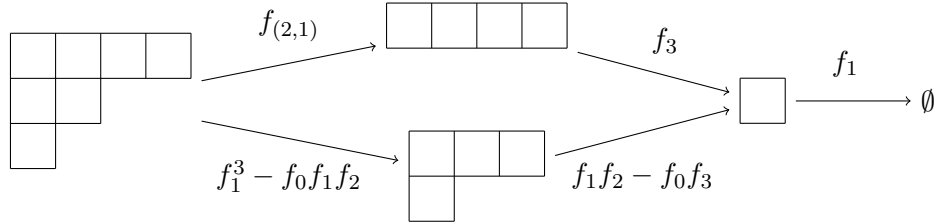
\begin{gather*}
  \centering
\begin{tikzpicture}
\node[] (1) at(1,0) {\begin{tikzpicture}[scale=0.6]
   \coordinate (Origin)   at (0,0);
    \coordinate (XAxisMin) at (0,0);
    \coordinate (XAxisMax) at (4,0);
    \coordinate (YAxisMin) at (0,-3);
    \coordinate (YAxisMax) at (0,0);
\draw [thin, black] (0,0) -- (4,0);
    \draw [thin, black] (0,-1) -- (4,-1);
    \draw [thin, black] (0,-2) -- (2,-2);
    \draw [thin, black] (0,-3) -- (1,-3);
    \draw [thin, black] (0,0) -- (0,-3);
    \draw [thin, black] (1,0) -- (1,-3);
    \draw [thin, black] (2,0) -- (2,-2);
    \draw [thin, black] (3,0) -- (3,-1);
    \draw [thin, black] (4,0) -- (4,-1);
    \end{tikzpicture}};
\node[] (2) at(6,1) {\begin{tikzpicture}[scale=0.6]
   \coordinate (Origin)   at (0,0);
    \coordinate (XAxisMin) at (0,0);
    \coordinate (XAxisMax) at (4,0);
    \coordinate (YAxisMin) at (0,-3);
    \coordinate (YAxisMax) at (0,0);
\draw [thin, black] (0,0) -- (4,0);
    \draw [thin, black] (0,-1) -- (4,-1);
    \draw [thin, black] (0,0) -- (0,-1);
    \draw [thin, black] (1,0) -- (1,-1);
    \draw [thin, black] (2,0) -- (2,-1);
    \draw [thin, black] (3,0) -- (3,-1);
    \draw [thin, black] (4,0) -- (4,-1);
    \end{tikzpicture}};
\node[] (3) at(6,-1) {\begin{tikzpicture}[scale=0.6]
   \coordinate (Origin)   at (0,0);
    \coordinate (XAxisMin) at (0,0);
    \coordinate (XAxisMax) at (4,0);
    \coordinate (YAxisMin) at (0,-3);
    \coordinate (YAxisMax) at (0,0);
\draw [thin, black] (0,0) -- (3,0);
    \draw [thin, black] (0,-1) -- (3,-1);
    \draw [thin, black] (0,-2) -- (1,-2);
    \draw [thin, black] (0,0) -- (0,-2);
    \draw [thin, black] (1,0) -- (1,-2);
    \draw [thin, black] (2,0) -- (2,-1);
    \draw [thin, black] (3,0) -- (3,-1);
    \end{tikzpicture}};
\node[] (4) at(9.8,0) {\begin{tikzpicture}[scale=0.6]
   \coordinate (Origin)   at (0,0);
    \coordinate (XAxisMin) at (0,0);
    \coordinate (XAxisMax) at (4,0);
    \coordinate (YAxisMin) at (0,-3);
    \coordinate (YAxisMax) at (0,0);
\draw [thin, black] (0,0) -- (1,0);
    \draw [thin, black] (0,0) -- (0,-1);
    \draw [thin, black] (1,0) -- (1,-1);
    \draw [thin, black] (0,-1) -- (1,-1);
    \end{tikzpicture}};
    \node[] (5) at(12,0) {$\emptyset$};
\draw[->] (1)--(2);
\draw[->] (1)--(3);
\draw[->] (2)--(4);
\draw[->] (3)--(4);
\draw[->] (4)--(5);
\node at(3.5,1) {$f_{(2,1)}$};
\node at(3.5,-1) {$f_1^3-f_0f_1f_2$};
\node at(8.5,0.8) {$f_{3}$};
\node at(8.2,-1) {$f_{1}f_2-f_0f_3$};
\node at(10.9,0.5) {$f_{1}$};
\end{tikzpicture}
\end{gather*}
Therefore
\begin{align*}
\zeta_{(3,3,1)}^{(4,2,1)}(q)&=\frac{1}{2(q-1)^{3}}\left(f_{(2,1)}f_3f_1+(f_1^3-f_0f_1f_2)(f_1f_2-f_0f_3)f_1\right)\\
&=8(q^4-6q^3+9q^2-6q+1).
\end{align*}
\end{exmp}

We consider some special cases.

($\uppercase\expandafter{\romannumeral1}$) When $\mu=(n),$ note that
\begin{align*}
wt_t(\la/\emptyset)=
\begin{cases}
f_{\la_1},& l(\la)=1;\\
2(-t)^{\lambda_2}f_{\lambda_1-\lambda_2},& l(\lambda)= 2;\\
0,& l(\lambda)\geq3.
\end{cases}
\end{align*}
Therefore
\begin{align}
\zeta_{(n)}^{\lambda}(q)=
\begin{cases}
2(-q)^{\lambda_2}(\lambda_1-\lambda_2)_q,& l(\lambda)\leq 2;\\
0,& l(\lambda)\geq3.
\end{cases}
\end{align}
So $\tilde{g}_n=\sum\limits_{i=0}^{\frac{n-1}{2}}(-q)^i(n-2i)_qQ_{(n-k,k)}.$

($\uppercase\expandafter{\romannumeral2}$) When $\mu=(1^n),$ then
\begin{align}
\zeta^{\lambda}_{(1^n)}(q)=2^{-\epsilon(\lambda)}(q-1)^{-n}\sum_{\underline{\lambda}}wt(\underline{\lambda})
\end{align}
summed over all ordered strict partition-value functions $\underline{\lambda}=(\lambda,\lambda^{(1)},\lambda^{(2)},\cdots,\lambda^{(n)})$ such that $\lambda^{(i)}\in\mathcal{SP}_{n-i},$ $i=1,2,\ldots,n.$ By \eqref{e:gds}, we have $wt_q(\lambda/\mu)=f_1$ provided that $|\lambda/\mu|=1.$ In this case, the number of $\underline{\lambda}$ is $g^{\lambda},$ and they have the same weight $f^n_{1}.$ Therefore,
\begin{align}
\zeta^{\lambda}_{(1^n)}(q)=2^{n-\epsilon(\lambda)}g^{\lambda}=\frac{2^{n-\epsilon(\lambda)}n!}{\la_1!\cdots\la_l!}\prod_{1\leq i<j\leq l}\frac{\la_i-\la_j}{\la_i+\la_j}.
%\frac{n!}{\prod_{x\in \lambda}h(x)}.
\end{align}

($\uppercase\expandafter{\romannumeral3}$) When $\mu=(k,1^{n-k}),$ we have
\begin{align*}
G^{\lambda}_{(k,1^{n-k})}(t)&=\langle g_kg_1\cdots g_1, Q_{\lambda}.1 \rangle\\
&=\langle g_1^{n-k}, \sum_{\nu} wt_t(\lambda/\nu)Q_{\nu}.1 \rangle\\
&=\sum_{\nu} wt_t(\lambda/\nu)f_1^{n-k}g^{\nu}.
\end{align*}
Therefore
\begin{align}
\zeta^{\lambda}_{(k,1^{n-k})}(q)=\frac{2^{n-k-\epsilon(\lambda)}}{q-1}\sum_{\nu} wt_q(\lambda/\nu)g^{\nu}.
\end{align}
summed over all strict partitions $\nu\subset\la$ such that $\la/\nu$ is a $k$-generalized double strip.
\section{Spin bitrace: the second orthogonality relation for $\zeta^{\lambda}_{\mu}(q)$}
In this section, we introduce the spin bitrace of the Hecke-Clifford algebra $\mathcal{H}^c_{n}$ as an analogue of the bitrace for the Hecke algebra \cite{HLR}.

Recall that $\mathcal{H}^c_{n}$ has a linear basis consisting of $T_{\sigma}C_{I},$ $\sigma\in S_n, I\subset [n].$ For any $x, y\in S_n,$ $I, J\subset [n],$ we define the spin bitrace of $T_xC_{I}$ and $T_yC_{J}$ as follows:
\begin{align}
\sbtr(T_xC_{I}, T_yC_{J})\doteq\sum_{\mbox{\tiny$\begin{array}{c}
z\in S_n\\
K\subset [n]\end{array}$}}T_xC_{I}T_zC_{K}T_{y}C_{J}\mid_{T_zC_{K}}
\end{align}
where $T_xC_{I}T_zC_{K}T_{y}C_{J}\mid_{T_zC_{K}}$ denotes the coefficient of the basis element $T_zC_{I}$ in the expansion of $T_xC_{I}T_zC_{K}T_{y}C_{J}.$

Let $L_{x,I}$ and $R_{y,I}$ be the left and the right multiplication of $\mathcal{H}^c_{n}$ by $T_xC_{I}$ on $\mathcal{H}^c_{n}$ respectively. If $x, y \in S_n$, $I, J\subset [n],$ then $L_{x,I}$ and $R_{y,J}$ commute and
\begin{align*}
\sbtr(T_xC_{I}, T_yC_{J})= Tr(L_{x,I}R_{y,J}).
\end{align*}

 Note that irreducible modules of $\mathcal{H}^c_{n}$ are divided into two types: type $M$ and type $Q$ due to the general property of supermodules \cite{Jo}. For $\lambda\in\mathcal{SP}_n,$ the type of irreducible module $V_{\lambda}$ is denoted by $\delta(\lambda)$: $\delta(\lambda)=0$ or $1$ if $V_{\lambda}$ is of type $M$ or $Q$ respectively. Moreover, it is known that \cite{WW}
\begin{align}\label{e:Hom}
{\rm dimHom}_{\mathcal{H}^c_n}(V_{\lambda}, V_{\mu})=2^{\delta(\lambda)}\delta_{\lambda\mu}.
\end{align}

If $V_i$ ($i=1, 2$) are left $\mathcal{H}^c_{n}$ simple supermodules, then $V_1\otimes V_2$ remains an irreducible simple $\mathcal{H}^c_{n}\otimes \mathcal{H}^c_{n}$-supermodule if one of $V_i$ is of type $M$ or $V_1\otimes V_2$ decomposes into $U\oplus U$ of irreducible supermodules of type $M$ if both $V_i$ are of type $Q$. In the latter case, we denote the irreducible summand $U$ of $V_1\otimes V_2$ as $2^{-1}V_1\otimes V_2$.
It is known that all irreducible $\mathcal{H}^c_{n}\otimes \mathcal{H}^c_{n}$-supermodules can be realized this way.

Note that the action of left $\mathcal{H}^c_n$-module commutes with that of right $\mathcal{H}^c_n$-module.
As a $\mathcal{H}^c_{n}\times \mathcal{H}^c_{n}$-bimodule, $\mathcal{H}^c_{n}$ decomposes into a direct sum of simple $\mathcal{H}^c_{n}\times \mathcal{H}^c_{n}$-supermodules.
It follows from \eqref{e:Hom} that the irreducible $\mathcal{H}^c_{n}\times \mathcal{H}^c_{n}$-bimodule summands of $\mathcal{H}^c_{n}$
are either $V_{\lambda}\otimes V^{\lambda}$ if one of the factors is of type $M$ or
$2^{-1}V_{\lambda}\otimes V^{\lambda}$ if both factors are in type $Q$, where $V_{\lambda}$ (resp. $V^{\lambda}$) is the irreducible left (resp. right) $\mathcal{H}^c_n$-module labeled by $\lambda$.
 Consequently by the structure theorem \cite[Prop. 2.11]{Jo} and \eqref{e:Hom} it follows that
%It follows from the double centralizer theory and \eqref{e:Hom} that
\begin{align}\label{e:Hdecomposition}
\mathcal{H}^c_n\cong\bigoplus_{\lambda\in\mathcal{SP}_n} 2^{-\delta(\lambda)}V_{\lambda}\otimes V^{\lambda}
\end{align}
as $\mathcal{H}^c_n\times \mathcal{H}^c_n$-bimodules. We remark that can also see that
$\mathcal{H}^c_n\cong\bigoplus_{\lambda\in\mathcal{SP}_n} 2^{-\delta(\lambda)}(\dim V_{\lambda}) V_{\lambda}$ as left $\mathcal{H}^c_n$-modules
by using \eqref{e:Hom}.
%where $V_{\lambda}$ (resp. $V^{\lambda}$) is the irreducible left (resp. right) $\mathcal{H}^c_n$-module labeled by $\lambda$.

Note that $V^{\lambda}\cong V_{\lambda}$. Taking trace on both sides of \eqref{e:Hdecomposition} gives
\begin{align}
\sbtr(T_xC_{I}, T_yC_{J})=\sum_{\lambda\in\mathcal{SP}_n}2^{-\delta(\lambda)}\zeta^{\lambda}(T_xC_{I})\zeta^{\lambda}(T_yC_{J}).
\end{align}

Note that any character of $\mathcal{H}^c_n$ is completely determined by its values on the elements $T_{\gamma_{\mu}}, \mu\in\mathcal{OP}_n$ (see \cite{WW} Corollary 4.9). Therefore, we can define
\begin{align*}
\sbtr(\mu, \nu)\doteq btr(T_{\gamma_{\mu}} ,T_{\gamma_{\nu}})
\end{align*}
for any two odd compositions $\mu, \nu$ of $n$.

We can calculate the spin bitrace using vertex operators. Recall the Frobenius formula \cite{WW}:
\begin{align*}
\tilde{g}_{\mu}=\sum_{\lambda\in\mathcal{SP}_n}2^{-\frac{l(\lambda)+\delta(\lambda)}{2}}Q_{\lambda}\zeta^{\lambda}(T_{\gamma_{\mu}}).
\end{align*}
Therefore we have
\begin{align}\label{e:gg}
\sbtr(\mu, \nu)=\sbtr(T_{\gamma_{\mu}} ,T_{\gamma_{\nu}})=\sum_{\lambda\in\mathcal{SP}_ n}2^{-\delta(\lambda)}\zeta^{\lambda}(T_{\gamma_{\mu}})\zeta^{\lambda}(T_{\gamma_{\nu}})=\langle \tilde{g}_{\mu}, \tilde{g}_{\nu} \rangle.
\end{align}

The Hecke-Clifford algebra $\mathcal{H}^c_n$ is a $q$-deformation of the algebra $\mathfrak{H}^c_n=C_n\rtimes S_n$, whose category of finite dimensional modules is known to be equivalent to that of Schur's spin symmetric group algebra \cite{Y}. The specialization at $q=1$ of the irreducible character $\zeta^{\lambda}_{\mu}(q)$ recovers the character $\chi^{\lambda}_{\mu}$ of $\mathfrak{H}^c_n,$.

Note that
\begin{align*}
\tilde{g}_{k}\mid_{q=1}=
\begin{cases}
2p_k& \text{if $k$ is odd}\\
0&\text{if $k$ is even}
\end{cases}
\end{align*}
for any $k\geq 1.$

Thus, the specialization at $q=1$ of \eqref{e:gg} gives
\begin{align}
\sbtr(\mu, \nu)\mid_{q=1}=\sum_{\lambda\vdash_{s} n}2^{-\delta(\lambda)}\chi^{\lambda}_{\mu}\chi^{\lambda}_{\nu}=\langle \tilde{g}_{\mu}, \tilde{g}_{\nu} \rangle\mid_{q=1}=\langle 2^{l(\mu)}p_{\mu}, 2^{l(\nu)}p_{\nu} \rangle=2^{l(\mu)}z_{\mu}\delta_{\mu,\nu}.
\end{align}
So the spin bitrace can be viewed as a deformation of the second orthogonality relation for $\zeta^{\lambda}_{\mu}(q).$

If we denote $T^{\mu}_{\nu}(t)=\langle g_{\mu}, g_{\nu} \rangle,$ then
\begin{align}
\sbtr(\mu, \nu)=\langle \tilde{g}_{\mu}, \tilde{g}_{\nu} \rangle=\frac{1}{(q-1)^{\l(\mu)+l(\nu)}}\langle g_{\mu}, g_{\nu} \rangle=\frac{1}{(q-1)^{\l(\mu)+l(\nu)}}T^{\mu}_{\nu}(q).
\end{align}
Next, we compute $T^{\mu}_{\nu}(t)$ using the technique developed in the previous sections and derive a general combinatorial formula. As an application, we will compute the regular character $\zeta^{reg}.$

\begin{prop}\label{t:g^*g} For any $n, m\in \mathbb{Z},$ the components $g_n$ and $g^*_m$ satisfy the following relation:
\begin{align}\label{e:g^*g}
\begin{split}
g^*_{m}g_{n}&=g_{n}g^*_{m}+(t-1)^2(g^*_{m-1}g_{n-1}+g_{n-1}g^*_{m-1})+2t(t^2-t+1)(g^*_{m-2}g_{n-2}-g_{n-2}g^*_{m-2})\\
&+t^2(t-1)^2(g^*_{m-3}g_{n-3}+g_{n-3}g^*_{m-3})-t^4(g^*_{m-4}g_{n-4}-g_{n-4}g^*_{m-4}).
\end{split}
\end{align}
\end{prop}

\begin{proof} By vertex operator calculus,
%Using the technique of vertex operators,
we have
\begin{align*}
\mathbf g^*(z)\mathbf g(w)=\mathbf g(w)\mathbf g^*(z)\frac{z+t^2w}{z-t^2w}(\frac{z-tw}{z+tw})^2\frac{z+w}{z-w}.
\end{align*}
Then \eqref{e:g^*g} follows by comparing coefficients.
\end{proof}

For convenience, we define for $n\geq 0$
\begin{align}\label{e:alpha}
\alpha_n(t)=\sum_{\rho\vdash_o n}\frac{2^{l(\rho)}z_{\rho}}{z^2_{\rho}(t)}=\sum_{\rho\vdash_o n}\frac{2^{l(\rho)}\prod(t^{\rho_i}-1)^2}{z_{\rho}}.
\end{align}
By convention $\alpha_n(t)=0$ for $n < 0$ and $\alpha_0(t)=1$. Sometimes we will abbreviate $\alpha_n$ by $\alpha_n(t).$ We also denote $\alpha_{\lambda}(t)=\prod\alpha_{\lambda_i}(t)$ for a composition $\lambda$. The first few terms can be listed as follows.
\begin{align*}
\alpha_0(t)=1, \alpha_{1}(t)=2(t-1)^2,  \alpha_2(t)=2(t-1)^4, \cdots
\end{align*}
We have the following iterative relations for $\alpha_n.$
\begin{lem}\label{t:alpharelation}
The following relations of $\alpha_n$ hold:
\begin{align}
\begin{split}\label{e:alpharelation1}
\alpha_1&=2(t-1)^2\alpha_0, \quad \alpha_{2}=(t-1)^2\alpha_1,\\
\alpha_3&=(t-1)^2\alpha_2+2t(t^2-t+1)\alpha_1+2t^2(t-1)^2\alpha_0,\\
\alpha_4&=(t-1)^2\alpha_3+2t(t^2-t+1)\alpha_2+t^2(t-1)^2\alpha_1,\\
\end{split}\\\label{e:alpharelation2}
\alpha_n&=(t-1)^2\alpha_{n-1}+2t(t^2-t+1)\alpha_{n-2}+t^2(t-1)^2\alpha_{n-3}-t^4\alpha_{n-4}  \quad (n\geq 5).
\end{align}
\end{lem}
\begin{proof}
It is easy to check \eqref{e:alpharelation1}. So we just verify \eqref{e:alpharelation2}. Recall that
\begin{align*}
g_n=\sum_{\rho\in\mathcal{OP}_n}\frac{(-2)^{l(\rho)}}{z_{\rho}(t)}p_{\rho}.
\end{align*}
Then we have
\begin{align*}
\langle g_n, g_n \rangle=\left\langle \sum_{\rho\in\mathcal{OP}_n}\frac{(-2)^{l(\rho)}}{z_{\rho}(t)}p_{\rho}, \sum_{\tau\in\mathcal{OP}_n}\frac{(-2)^{l(\tau)}}{z_{\tau}(t)}p_{\tau} \right\rangle=\sum_{\rho\in\mathcal{OP}_n}\frac{2^{l(\rho)}z_{\rho}}{z^2_{\rho}(t)}=\alpha_n.
\end{align*}
Thus, $\alpha_n=g^*_{n}g_n.$ Setting $m=n\geq 5$ in \eqref{e:g^*g}, we obtain \eqref{e:alpharelation2}.
\end{proof}

\begin{prop}\label{t:g*g}
Let $k$ be a non-negative integer and $\lambda$ a composition, then we have
\begin{align}\label{e:g*g}
g^*_kg_{\lambda}=\sum_{\tau\models k}\alpha_{\tau}g_{\lambda-\tau}
\end{align}
\end{prop}
\begin{proof}
This is checked by induction on $k+|\lambda|.$ The initial step is trivial. Assume that \eqref{e:g*g} holds for $<k+|\lambda|.$ We show it is true when $k+|\lambda|.$ By \eqref{e:g^*g}, we have (denoting the LHS of \eqref{e:g*g} by I)
\begin{align*}
I=g^*_{k}g_{\lambda}&=g_{\lambda_1}g^*_kg_{\lambda^{(1)}}+(t-1)^2(g^*_{k-1}g_{\lambda_1-1}g_{\lambda^{(1)}}+g_{\lambda_1-1}g^*_{k-1}g_{\lambda^{(1)}})\\
&+2t(t^2-t+1)(g^*_{k-2}g_{\lambda_1-2}g_{\lambda^{(1)}}-g_{\lambda_1-2}g^*_{k-2}g_{\lambda^{(1)}})\\
&+t^2(t-1)^2(g^*_{k-3}g_{\lambda_1-3}g_{\lambda^{(1)}}+g_{\lambda_1-3}g^*_{k-3}g_{\lambda^{(1)}})\\
&-t^4(g^*_{k-4}g_{\lambda_1-4}g_{\lambda^{(1)}}-g_{\lambda_1-4}g^*_{k-4}g_{\lambda^{(1)}}).
\end{align*}
By induction, we have
\begin{align}\label{e:I}
\begin{split}
I&=\sum_{\tau\models k}\alpha_{\tau}g_{\lambda_1}g_{\lambda^{(1)}-\tau}+(t-1)^2\left(\sum_{\tau\models k-1}\alpha_{\tau}g_{\lambda_1-1-\tau_1}g_{\lambda^{(1)}-\tau^{(1)}}+\sum_{\tau\models k-1}\alpha_{\tau}g_{\lambda_1-1}g_{\lambda^{(1)}-\tau}\right)\\
&+2t(t^2-t+1)\left(\sum_{\tau\models k-2}\alpha_{\tau}g_{\lambda_1-2-\tau_1}g_{\lambda^{(1)}-\tau^{(1)}}-\sum_{\tau\models k-2}\alpha_{\tau}g_{\lambda_1-2}g_{\lambda^{(1)}-\tau}\right)\\
&+t^2(t-1)^2\left(\sum_{\tau\models k-3}\alpha_{\tau}g_{\lambda_1-3-\tau_1}g_{\lambda^{(1)}-\tau^{(1)}}+\sum_{\tau\models k-3}\alpha_{\tau}g_{\lambda_1-3}g_{\lambda^{(1)}-\tau}\right)\\
&-t^4\left(\sum_{\tau\models k-4}\alpha_{\tau}g_{\lambda_1-4-\tau_1}g_{\lambda^{(1)}-\tau^{(1)}}-\sum_{\tau\models k-4}\alpha_{\tau}g_{\lambda_1-4}g_{\lambda^{(1)}-\tau}\right).
\end{split}
\end{align}

Let us pause to introduce a notation:
\begin{align*}
I(i,j):=\sum_{\mbox{\tiny$\begin{array}{c}
\tau\models k\\
\tau_1=j\end{array}$}}\alpha_i\alpha_{\tau^{(1)}}g_{\lambda-\tau}.
\end{align*}
It is sufficient to prove
\begin{align*}
I=\sum_{i\geq0}I(i,i).
\end{align*}
By Lemma \ref{t:alpharelation}, we immediately obtain the following relations of $I(i,j)$:
\begin{align*}
2(t-1)^2I(0,1)=I(1,1)&; \quad (t-1)^2I(1,2)=I(2,2);\\
(t-1)^2I(2,3)+2t(t^2-t+1)&I(1,3)+2t^2(t-1)^2I(0,3)=I(3,3);\\
(t-1)^2I(3,4)+2t(t^2-t+1)&I(2,4)+t^2(t-1)^2I(1,4)=I(4,4);\\
(t-1)^2I(i-1,i)+2t(t^2-t+1)I(i-2,i)&+t^2(t-1)^2I(i-3,i)-t^4I(i-4,i)=I(i,i)  \quad (i\geq 5).
\end{align*}
Now, let's consider \eqref{e:I} term by term. It's clear that $\sum_{\tau\models k}\alpha_{\tau}g_{\lambda_1}g_{\lambda^{(1)}-\tau}=I(0,0).$ And,
\begin{align*}
&\sum_{\tau\models k-1}\alpha_{\tau}g_{\lambda_1-1-\tau_1}g_{\lambda^{(1)}-\tau^{(1)}}+\sum_{\tau\models k-1}\alpha_{\tau}g_{\lambda_1-1}g_{\lambda^{(1)}-\tau}\\
=&\sum_{\mbox{\tiny$\begin{array}{c}
\tau\models k\\
\tau_1\geq1\end{array}$}}\alpha_{\tau_1-1}\alpha_{\tau^{(1)}}g_{\lambda-\tau}+\sum_{\mbox{\tiny$\begin{array}{c}
\tau\models k\\
\tau_1=1\end{array}$}}\alpha_0\alpha_{\tau^{(1)}}g_{\lambda-\tau}\\
=&\sum_{i\geq1}I(i-1,i)+I(0,1).
\end{align*}
Similarly, We have
\begin{align*}
&\sum_{\tau\models k-2}\alpha_{\tau}g_{\lambda_1-2-\tau_1}g_{\lambda^{(1)}-\tau^{(1)}}-\sum_{\tau\models k-2}\alpha_{\tau}g_{\lambda_1-2}g_{\lambda^{(1)}-\tau}=\sum_{i\geq3}I(i-2,i).\\
&\sum_{\tau\models k-3}\alpha_{\tau}g_{\lambda_1-3-\tau_1}g_{\lambda^{(1)}-\tau^{(1)}}-\sum_{\tau\models k-3}\alpha_{\tau}g_{\lambda_1-3}g_{\lambda^{(1)}-\tau}=\sum_{i\geq3}I(i-3,i)+I(0,3).\\
&\sum_{\tau\models k-4}\alpha_{\tau}g_{\lambda_1-4-\tau_1}g_{\lambda^{(1)}-\tau^{(1)}}-\sum_{\tau\models k-4}\alpha_{\tau}g_{\lambda_1-4}g_{\lambda^{(1)}-\tau}=\sum_{i\geq5}I(i-4,i).\\
\end{align*}
It follows from the above relations of $I(i,j)$ that
\begin{align*}
I&=I(0,0)+(t-1)^2\left(I(0,1)+\sum_{i\geq2}I(i-1,i)\right)+2t(t^2-t+1)\sum_{i\geq3}I(i-2,i)\\
&+t^2(t-1)^2\left(I(0,3)+\sum_{i\geq3}I(i-3,i)\right)-t^4\sum_{i\geq5}I(i-4,i)\\
&=\sum_{i\geq0}I(i,i).
\end{align*}
\end{proof}

Now the following is clear.
\begin{thm}\label{t:T}
Given two compositions $\mu, \nu\models n$, then
\begin{align}
T^{\mu}_{\nu}(t)=\sum_{\tau\models \mu_1}\alpha_{\tau}T^{\mu^{(1)}}_{\nu-\tau}(t).
\end{align}
\end{thm}

Let $A=(a_{ij})_{l\times r}$ be an $l\times r$ matrix, where $a_{ij}$ are non-negative integers. We introduce the spin weight of $A$ by
\begin{align}
\tilde{wt}(A):=\prod_{i,j}\alpha_{a_{ij}}.
\end{align}
Now we come to a general combinatorial formula for $\sbtr(\mu,\nu).$
\begin{thm}\label{t:btr}
Let $\mu, \nu\models n,$ $\mu=(\mu_1,\mu_2,\ldots,\mu_l)$ and $\nu=(\nu_1,\nu_2,\ldots,\nu_r),$ then
\begin{align}
\sbtr(\mu,\nu)=(q-1)^{-l(\mu)-l(\nu)}\sum_{A}\tilde{wt}(A)
\end{align}
where the sum runs over all $l\times r$ nonnegative integer matrices $A$ with row sums $\mu_1,\mu_2,\ldots,\mu_l$ and column sums $\nu_1,\nu_2,\ldots,\nu_r.$
\end{thm}

As an application, we compute the regular character $\zeta^{reg}$ of $\mathcal{H}^c_{n}.$
\begin{cor}\label{t:reg}
Let $\mu\in\mathcal{OP}_n,$ then the trace of the regular representation of the Hecke-Clifford algebra $\mathcal{H}^c_n$ is given by
\begin{align}\label{e:reg}
\zeta^{reg}(\mu)=2^n(q-1)^{n-l(\mu)}
%\frac{n!}{\mu_1!\cdots\mu_l!}\prod_{1\leq i<j\leq l}\frac{\mu_i-\mu_j}{\mu_i+\mu_j}
\frac{n!}{\prod_{i}\mu_i!}.
\end{align}
\end{cor}
\begin{proof}
By definition, we have $\zeta^{reg}(\mu)=\zeta^{reg}(T_{\gamma_{\mu}})=\sbtr(\mu,(1^n)).$ Using Theorem \ref{t:btr}, we find that the sum in $\sbtr(\mu,(1)^n)$ runs over all matrices with entities $0$ or $1$ such that each column contains exactly one $1$ and the $i$-th row contains exactly $\mu_i$ $1'$s. Then \eqref{e:reg} follows from the fact that the number of these matrices is $\frac{n!}{\prod_{i}\mu_i!}$ and they have the same spin weight $\alpha_1^n=2^n(q-1)^{2n}.$
\end{proof}

{\bf Tables for $\zeta^{\lambda}_{\mu}(q)$}\\

\begin{table}[H]

 \centering

\caption{\label{tab:3}n=3}

 \begin{tabular}{|c|c|c|c|}

  \hline

 \tabincell{c}{$\mu\backslash \lambda$ } & $(3)$ & $(2,1)$ \\

  \hline

$(3)$ & \tabincell{c}{$2(3)_q$} & \tabincell{c}{$-2q$}  \\

\hline

$(1^3)$   & \tabincell{c}{$2^3$}    & \tabincell{c}{$2^2$}      \\

  \hline

 \end{tabular}

\end{table}

\begin{table}[H]

 \centering

\caption{\label{tab:4}n=4}

 \begin{tabular}{|c|c|c|c|}

  \hline

 \tabincell{c}{$\mu\backslash \lambda$ } & $(4)$ & $(3,1)$ \\

  \hline

$(3,1)$ & \tabincell{c}{$2^2(3)_q$} & \tabincell{c}{$2(q^2-3q+1)$}  \\

\hline

$(1^4)$   & \tabincell{c}{$2^4$}    & \tabincell{c}{$2^4$}      \\

  \hline

 \end{tabular}

\end{table}

\begin{table}[H]

 \centering

\caption{\label{tab:5}n=5}

 \begin{tabular}{|c|c|c|c|c|}

  \hline

 \tabincell{c}{$\mu\backslash \lambda$ } & $(5)$ & $(4,1)$ & $(3,2)$  \\

  \hline

$(5)$ & \tabincell{c}{$2(5)_q$} & \tabincell{c}{$-2q(3)_q$} & \tabincell{c}{$2q^2$} \\

\hline

$(3,1^2)$ & \tabincell{c}{$2^3(3)_q$}  & \tabincell{c}{$2^3(q-1)^2$}  & \tabincell{c}{$2^2(q^2-3q+1)$}  \\

  \hline

$(1^5)$   & \tabincell{c}{$2^5$}    & \tabincell{c}{$3\cdot2^4$}    & \tabincell{c}{$2^5$}     \\

  \hline

 \end{tabular}

\end{table}

\begin{table}[H]

 \centering

\caption{\label{tab:6}n=6}

 \begin{tabular}{|c|c|c|c|c|c|}

  \hline

 \tabincell{c}{$\mu\backslash \lambda$ } & $(6)$ & $(5,1)$ & $(4,2)$ & $(3,2,1)$ \\

  \hline

$(5,1)$ & \tabincell{c}{$2^2(5)_q$} & \tabincell{c}{$2(q^4-3q^3+3q^2-3q+1)$} & \tabincell{c}{$-2^2q(q-1)^2$} & \tabincell{c}{$2^2q^2$}  \\

\hline

$(3,3)$ & \tabincell{c}{$2^2(3)_q(3)_q$}  & \tabincell{c}{$2^2(3)_q(q^2-3q+1)$}  & \tabincell{c}{$2^2(q^4-4q^3+7q^2-4q+1)$}  & \tabincell{c}{$-2^3q(3)_q$}  \\

  \hline

$(3,1^3)$   & \tabincell{c}{$2^4(3)_q$}    & \tabincell{c}{$2^3(3q^2-5q+3)$}    & \tabincell{c}{$2^3(3q^2-7q+3)$}    & \tabincell{c}{$2^3(q^2-3q+1)$}   \\
\hline

$(1^6)$   & \tabincell{c}{$2^6$}    & \tabincell{c}{$2^7$}    & \tabincell{c}{$5\cdot2^5$}    & \tabincell{c}{$2^6$}   \\
  \hline
 \end{tabular}

\end{table}

\begin{table}[H]

 \centering

\caption{\label{tab:7}n=7}

 \begin{tabular}{|c|c|c|c|c|c|c|c|}

  \hline

 \tabincell{c}{$\mu\backslash \lambda$ } & $(7)$ & $(6,1)$ & $(5,2)$ & $(4,3)$ & $(4,2,1)$ \\

  \hline

$(7)$ & \tabincell{c}{$2(7)_q$} & \tabincell{c}{$-2q(5)_q$} & \tabincell{c}{$2q^2(3)_q$} & \tabincell{c}{$-2q^3$} & \tabincell{c}{$0$} \\

\hline

$(5,1,1)$ & \tabincell{c}{$2^3(5)_q$} & \tabincell{c}{$2^3(q-1)(4)_q$} & \tabincell{c}{$2^2(q^4-5q^3$\\$+7q^2-5q+1)$} & \tabincell{c}{$-2^3q(q-1)^2$} & \tabincell{c}{$-2^3q(q^2-3q+1)$}  \\

\hline

$(3,3,1)$ & \tabincell{c}{$2^3(3)_q(3)_q$}  & \tabincell{c}{$2^2(3)_q(3q^2-7q+3)$}  & \tabincell{c}{$2^4(q-1)^4$} & \tabincell{c}{$2^3(q^4-4q^3$\\$+7q^2-4q+1)$} & \tabincell{c}{$2^3(q^4-6q^3+9q^2-6q+1)$}   \\

\hline

$(3,1^4)$ & \tabincell{c}{$2^5(3)_q$}  & \tabincell{c}{$2^5(2q^2-3q+2)$}  & \tabincell{c}{$3\cdot2^5(q-1)^2$} & \tabincell{c}{$2^4(3q^2-7q+3)$} & \tabincell{c}{$2^5(2q^2-5q+2)$}  \\

  \hline

$(1^7)$   & \tabincell{c}{$2^7$}    & \tabincell{c}{$5\cdot2^6$}    & \tabincell{c}{$2^6\cdot3^2$}  & \tabincell{c}{$5\cdot2^6$}  & \tabincell{c}{$7\cdot2^6$}   \\
  \hline
 \end{tabular}

\end{table}

%\bigskip

\vskip30pt \centerline{\bf Acknowledgments}
We would like to thank Zhijun Li for interesting discussions on Pfaffians.
The work is partially supported by
Simons Foundation grant No. 523868 and NSFC grant No. 12171303.
\bigskip

\bibliographystyle{plain}

\end{document}